\documentclass{article}

\usepackage[mathscr]{euscript}
\usepackage[margin=1in]{geometry}
\usepackage{amsmath, amsthm, amsfonts, amssymb}
\usepackage{enumerate}
\usepackage[all, cmtip]{xy}

\newtheorem{theorem}{Theorem}[section]
\newtheorem{proposition}[theorem]{Proposition}
\newtheorem{corollary}[theorem]{Corollary}

\newtheorem{lemma}[theorem]{Lemma}

\theoremstyle{definition}
\newtheorem{definition}[theorem]{Definition}
\newtheorem{example}[theorem]{Example}

\newtheorem{notation}[theorem]{Notation}

\theoremstyle{remark}
\newtheorem{remark}[theorem]{Remark}

\newcommand{\calV}{\mathscr{V}}
\newcommand{\calY}{\mathscr{Y}}
\newcommand{\calX}{\mathscr{X}}
\newcommand{\calS}{\mathscr{S}}
\newcommand{\calZ}{\mathscr{Z}}

\newcommand{\cref}[1]{\ref{#1}}
\newcommand{\Cref}[1]{\ref{#1}}
\newcommand{\tor}{\mathrm{Tor}}
\newcommand{\G}{\mathbb{G}}
\newcommand{\ind}{\mathrm{Ind}}
\newcommand{\calR}{\mathcal{R}}
\newcommand{\calB}{\mathcal{B}}
\newcommand{\cha}{\mathrm{char}\,}
\newcommand{\calU}{\mathcal{U}}
\renewcommand{\hom}{\mathrm{Hom}}
\newcommand{\into}{\hookrightarrow}
\newcommand{\ol}[1]{\overline{#1}}
\newcommand{\iffw}{if and only if }
\newcommand{\irr}{\mathrm{Irr}}
\newcommand{\spec}{\mathrm{Spec}\,}
\newcommand{\shom}{\mathcal{H}\mathrm{om}}
\newcommand{\sym}{\mathrm{Sym}}

\newcommand{\bul}{\bullet}
\newcommand{\codim}{\mathrm{codim}\,}
\newcommand{\fix}{\mathrm{Fix}}
\newcommand{\bs}{\backslash}
\newcommand{\fk}[1]{\mathfrak{#1}}
\newcommand{\oii}{^{-1}}                                             
\newcommand{\ul}[1]{\underline{#1}}
\newcommand{\calT}{\mathcal{T}}
\newcommand{\calO}{\mathcal{O}}

\newcommand{\calM}{\mathcal{M}}
\newcommand{\calN}{\mathcal{N}}
\newcommand{\calH}{\mathcal{H}}
\newcommand{\calL}{\mathcal{L}}
\newcommand{\calJ}{\mathcal{J}}
\newcommand{\calF}{\mathcal{F}}
\newcommand{\calA}{\mathcal{A}}
\newcommand{\calC}{\mathcal{C}}

\newcommand{\N}{\mathbb{N}}
\newcommand{\C}{\mathbb{C}}
\newcommand{\A}{\mathbb{A}}
\newcommand{\Z}{\mathbb{Z}}
\newcommand{\Q}{\mathbb{Q}}

\newcommand{\map}{\mathrm{Map}}
\newcommand{\aut}{\mathrm{Aut\,}}
\newcommand{\emo}{\mathrm{End}\,}
\newcommand{\stab}{\mathrm{Stab}}
\newcommand{\rep}{\mathrm{Rep}\,}
\renewcommand{\tilde}[1]{\widetilde{#1}}
\renewcommand{\hat}[1]{\widehat{#1}}

\begin{document}
\title{Characteristic Classes of Cameral Covers}
\author{Edward Dewey\thanks{Supported by NSF grants DMS-1452276 and
	DMS-1502553} 
}

\maketitle

\begin{abstract}
\noindent We study the stack $\calM$ of cameral covers for a complex reductive group $G$, introduced by Donagi and Gaitsgory. We compute its cohomology ring $H^*(\calM,\Q)$. In the special case $G=GL(n)$, $\calM$ is the stack of spectral covers.  We also compute the cohomology ring of the stack of abstract regular $G$-Higgs bundles.
\end{abstract}
\section{Introduction}\label{chapIntroduction}   

Let $\fk{t}$ be a complex vector space and let $W$ be a finite subgroup of
$GL(\fk{t})$ generated by complex reflections.  

\begin{definition} Let $C \to S$ be an $S$-scheme with an action of $W$ on the
	left that fixes the map $C \to S$.  $C$ is a $(W,
	\fk{t})$-\emph{cameral cover} of $S$ if there exists an \'{e}tale cover
	$U \to S$, a map $U \to \fk{t}/W$, and a $W$-equivariant isomorphism 
	$U \times_{\fk{t}/W} \fk{t} \cong U \times_S C$ over $U$.  
\end{definition}

In other words, a cameral cover of $S$ is an $S$-scheme with $W$ action which
is locally isomorphic to a pullback of $\fk{t} \to \fk{t}/W$.  $W$ does not act on $S$ (or acts trivially).  These were
introduced by Donagi \cite{donagi}.  They subsume the better-known notion of
spectral cover.  Our cameral covers are slightly more general since we do not
require $W$ to be a real reflection group. In the case where $\fk{t}$ is the
Lie algebra of a maximal torus in a connected reductive affine algebraic group
$G$, and $W$ is its Weyl group, $(W, \fk{t})$ cameral covers are closely
related to (abstract regular) $G$-Higgs bundles (see \cite{DG} and section 7
below).  For further motivation see \cite{pantev}.  Let $\calM$ be the stack of
cameral covers, and if $(W ,\fk{t})$ was obtained from a group $G$, let $\calH$
be the stack of $G$-Higgs bundles.

Our goal is to define cohomological invariants (or ``characteristic classes'')
of cameral covers and Higgs bundles, by computing $H^*(\calM, \Q)$.  The Betti
numbers of $\calM$ are easy to find, so the main contributions of this paper
are (first) to compute the cup product, and (second) to identify the elements
of $H^*(\calM, \Q)$ concretely enough that, given a cameral cover $C\to S$, one
can hope to identify its characteristic classes in $H^*(S, \Q)$.   We do all
this by studying a stratification of $\calM$ into classifying spaces.  Using
the results of \cite{DG} we obtain a similar description of $H^*(\calH, \Q)$.  

Our answer is phrased in terms of hyperplane arrangements (see section 2 for
precise definitions).  Let $\calA$ be the hyperplane arrangement in $\fk{t}$ of
the reflecting hyperplanes for the action of $W$, and let
$L(\calA)$ be the intersection poset of $\calA$.  Given $X \in L(\calA)$, let
$\calA_X \subset \calA$ be the subarrangement consisting of hyperplanes that
contain $X$, and call $X$ irreducible if $\calA_X$ is irreducible.  Let $\irr
\calA_X$ be the set of irreducible components of $\calA_X$.  
Let $L^{\mu}(\calA)$ be the free
abelian monoid on the tuples $(X, \mu)$ where $X \in L(\calA)$ is irreducible,
$X \neq \fk{t}$, and $\mu \geq \codim X$ is an integer, modulo the relation $
\prod_{i=1}^l (X_i, \mu_i)  = \left( \bigcap_{i=1}^l X_i, \sum_{i=1}^l \mu_i
\right) $ whenever $\bigcap_{i=1}^l X_i$ is irreducible.  This monoid is graded
by $\deg (X, \mu) = 2\mu$.  Let $\kappa$ be a ring in which $\#W$ is a unit.
Our main result is \theoremstyle{plain} \newtheorem*{oneShot}{Theorem \ref{HM}}
\begin{oneShot} $H^*(\calM, \kappa) \cong \kappa[L^\mu(\calA)]^W$ as graded
rings.  \end{oneShot} Setting $(W, \fk{t}) = (\Sigma_n, \C^{\oplus n})$, this
implies that characteristic classes of rank-$n$ spectral covers are in
bijection with certain weighted partitions of $n$. 


Here is an outline of the text.  Section 2 recalls some definitions and
notations from the theory of hyperplane arrangements and describes
$L^{\mu}(\calA)$.  Section 3 introduces an induction operation on cameral
covers and uses it to understand cameral covers of $\spec \C$.  Section 4
studies the geometry of $\calM$ and introduces a stratification of $\calM$ by
classifying spaces.  Section 5 uses this stratification to compute the rational
cohomology ring of $\calM$.  As a demonstration that this ring structure is
potentially interesting, section 6 gives an analogue of the Whitney product
formula, describing the characteristic classes of induced cameral covers.
Section 7 computes the rational cohomology ring of $\calH$.
Section 8 describes some partial progress towards
$H^*(\calM, \Z)$ and an obstruction in integral cohomology to a cameral
cover being deformable to a less ramified cameral cover.  Finally, section 9
studies the K-theory of $\calM$ using methods analogous to those of section
5.

\textbf{Conventions:}  All schemes and stacks are over $\C$.  If $\calF$ is a
locally free sheaf then $|\calF|$ denotes its geometric realization $\spec
\sym^\bul \calF^\vee$.  Note that we have chosen the convention that makes
$|\cdot|$ covariant, and that makes a section of $\calF$ correspond to a
section of $|\calF|$.  

In this text we will ask topological questions about algebraic objects.  This
involves taking complex realizations, and we mostly leave this implicit.  In
particular, $H^*(X, \kappa)$ means cohomology of the complex realization.  

\textbf{Acknowledgments:}  The question answered here was raised by Dima
Arinkin and Roman Fedorov.  Thanks to them, as well as to Daniel Erman,
Steven Sam and the anonymous reviewers of the first 7 sections, for valuable suggestions and
corrections.
\section{Hyperplane arrangements}\label{chapHyperplanes} 

In this section we review some standard definitions and notations for hyperplane arrangements, and introduce the monoid $L^\mu(\calA)$.  For a good exposition of hyperplane arrangements see \cite{hyperplanes}.  

\begin{definition}
A \emph{hyperplane arrangement} $\calA$ in a vector space $\fk{t}$ is a (possibly empty) set of hyperplanes in $\fk{t}$.  A hyperplane arrangement is \emph{central} if all hyperplanes are subspaces.    
If $\calA$ is a hyperplane arrangement, its \emph{intersection poset} $L(\calA)$ is the set of all intersections of elements of $A$, reverse-ordered by inclusion.  
The \emph{codimension} $\codim \calA$ is the codimension in $\fk{t}$ of the intersection of all elements of $\calA$ (this definition is not standard). $\calA$ is \emph{essential} if $\codim \calA = \dim \fk{t}$, equivalently if $0 \in L(\calA)$.
\end{definition}

All hyperplane arrangements in this text will be finite, central, and will live in a finite-dimensional vector space.  $L(\calA)$ has joins, corresponding to intersections of subspaces of $\fk{t}$, and every element of $L(\calA)$ is a join of several elements of $\calA$ (i.e. $L(\calA)$ is \emph{atomic}).

\begin{example}  
\label{exGLn} Consider the action of the Weyl group $W = \Sigma_n$ of $GL_n$,
acting on the lie algebra $\fk{t} = \C^{\oplus n}$ of a maximal torus.  This is
an action by reflections and we can consider the hyperplane arrangement of
reflecting hyperplanes.  These are all of the form $a_{ij} = \langle e_1, \ldots,
\hat{e}_i \ldots \hat{e}_j \ldots e_n, e_i + e_j \rangle$.  Elements of
$L(\calA)$ correspond to decompositions of the set $[n] := \{1, \ldots, n\}$.
If $\tau$ and $\lambda$ are two decompositions then $\tau \leq \lambda$ \iffw
$\tau$ refines $\lambda$.  \end{example}

\begin{definition} If $f: \fk{t} \to \fk{t}'$ is a surjective linear
	transformation and $\calA$ is a hyperplane arrangement in $\fk{t}'$,
	write $f\oii\calA$ for the hyperplane arrangement $\{ f\oii(a) | a
	\in \calA \}$ in $\fk{t}$.  If $\calA$ and $\calA'$ are hyperplane
	arrangements in vector spaces $\fk{t}$ and $\fk{t}'$, write $\calA
	\oplus \calA'$ for the hyperplane arrangement $\pi_1\oii\calA \cup
	\pi_2\oii\calA'$ in $\fk{t} \oplus \fk{t}'$, where $\pi_1, \pi_2$ are the
	projections.  A hyperplane arrangement is \emph{irreducible} if it is
	not of this form for any nonempty arrangements $\calA$, $\calA'$.
\end{definition}

Note that an irreducible hyperplane arrangement need not be essential.

\begin{definition}
Any hyperplane arrangement $\calA$ can be decomposed as $\calA \cong \mathcal{E} \oplus \calB_1 \oplus \ldots \calB_n$ where $\mathcal{E}$ is an empty arrangement and the $\calB_i$ are irreducible and essential.  This decomposition is unique up to automorphisms and reorderings of the summands.  Call the $\calB_i$ the \emph{irreducible components} of $\calA$, and write $\irr \calA = \{\calB_1, \ldots, \calB_n\}$.  
\end{definition}

Note that an empty arrangement has no irreducible components.

\begin{definition}
For $X \in L(\calA)$, denote by $\calA_X \subset \calA$ the sub-arrangement consisting of hyperplanes containing $X$.  Call $X$ \emph{irreducible} if $\calA_X$ is an irreducible hyperplane arrangement.
\end{definition}


\begin{example}
Let $\calA$ be as in \cref{exGLn}.  Let $X \in L(\calA)$, corresponding to a decomposition of $[n] := \{1, \ldots n\}$.  Suppose the decomposition is $[n] = \coprod_{i=1}^l X_i$, where the $X_i$ are some disjoint subsets of $[n]$.  We have $\codim X = \sum_{i=1}^l \# X_i - 1$.  $\calA_X$ consists of all hyperplanes $a_{pq}$ such that $\{p, q\} \subseteq X_i$ for some $i$.  $X$ is irreducible \iffw $X_i$ is a singleton for all but one value of $i$.  Thus irreducible elements of $L(\calA)$ other than $\fk{t}$ correspond to subsets of $[n]$ of size at least $2$.
\end{example}

\begin{example}  
\label{exNonessential}
Let $\fk{t} = \C^4$ and let $\calA$ consist of the hyperplanes $a_1 = \langle e_1 + e_2, e_3, e_4 \rangle$, $a_2 = \langle e_1 - e_2, e_3, e_4 \rangle$, $a_3 = \langle e_1, e_2, e_4 \rangle$.  This is non-essential, with two irreducible components.  These are
\begin{enumerate}
\item $\langle e_1, e_2 \rangle$ with hyperplanes $a'_1 = \langle e_1 + e_2 \rangle$ and $a'_2 = \langle e_1 - e_2 \rangle$
\item $\langle e_3 \rangle$, with hyperplane $a'_3 = 0$
\end{enumerate}
Let $X = \langle e_3, e_4 \rangle $.  Then $\calA_X$ is the arrangement with underlying vector space $\C^4$ and with hyperplanes $a_1$ and $a_2$.  $X$ is not irreducible.
\end{example}

\begin{definition}
Let $F$ be the set of pairs $(X, \mu)$ where $X \in L(\calA)$ is irreducible,
$\mu \geq \codim X$ is an integer, and $X \neq \fk{t}$.    Let $L^\mu(\calA)$
be the free abelian monoid on $F$ (for which we use multiplicative notation), modulo the relation
\[ \prod_{i=1}^l (X_i, \mu_i)  = \left( \bigcap_{i=1}^l X_i, \sum_{i=1}^l \mu_i \right) \]
whenever $\bigcap_{i=1}^l X_i$ is irreducible.  $L^\mu(\calA)$ admits a $2\N$ grading by setting $\deg (X, \mu) = 2\mu$ and extending multiplicatively.  
\end{definition}

We will sometimes abuse notation by writing $X$ in place of $(X, \codim X)$.  $L^\mu(\calA)$ is generated as a unital monoid by the terms $a = (a, 1)$ where $a \in \calA$.  


\begin{example}
Again let $\fk{t} = \C^{\oplus n}$ and let $\calA$ be the hyperplane arrangement for the action of $\Sigma_n$.  Recall that irreducible elements of $L(\calA)$ other than the minimum element $\fk{t}$ correspond to subsets of $[n]$.  For $n=5$ the word $(\{1, 2\}, 1)(\{3, 4, 5\}, 7)$ cannot be simplified, while $(\{1, 2\}, 1)(\{2, 3, 4\}, 7) = (\{1, 2, 3, 4\}, 8)$.  For $n=2$, $L^{\mu}(\calA) \cong \N$, generated (in degree $2$) by $\{1, 2\}$.  For $n=3$, $L^{\mu}(\calA)$ is the abelian monoid on three letters $x = \{1,2\}$, $y = \{2, 3\}$ and $z = \{1, 3\}$ subject to the single relation $xy = yz = xz$.
\end{example}  

\begin{proposition}
\label{descripMonoid2}
Let $L^{ch}(\calA)$ be the free abelian monoid on all elements of $L(\calA)$ (including the reducible ones), modulo two relations:
\begin{enumerate}
\item For any $X, Y, Z \in L(\calA)$ with $X \cap Y = Z$, and such that $Z$ has the expected codimension $\codim Z = \codim X + \codim Y$, impose the relation $XY = Z$.
\item For any $a, b \in \calA$ and $X \in L(\calA)$ with $a, b$ in the same irreducible component of $X$, impose the relation $aX = bX$.
\end{enumerate}
Give it the grading $\deg X = 2 \, \codim X$.  Then $L^{ch}(\calA) \cong L^{\mu}(\calA)$ as graded monoids.  If $X \in L(\calA)$ is irreducible then this isomorphism sends $X$ to $(X, \codim X)$.
\end{proposition}

\begin{proof}
We will describe a monoid morphism $\phi: L^\mu(\calA) \to L^{ch}(\calA)$.
Suppose $X \in L(\calA)$ is irreducible.  Let $a$ be a hyperplane in $\calA_X$.
We let $\phi(X, \mu) = a^{\mu - \codim X} X$, and extend multiplicatively.  The
second relation on $\calL^{ch}(\calA)$ ensures that this does not depend on the
choice of $a$, and it is easy to check that $\phi$ respects the defining
relation of $L^\mu(\calA)$, so $\phi$ is well-defined.  

We will define an inverse function $\psi$.  Suppose $X \in L(\calA)$ (not necessarily irreducible), and view it as an element of $L^{ch}(\calA)$.  Then $X$ may be written $X = \prod_{i=1}^l X_i$ where $X_1, \ldots, X_l$ are the irreducible components of $X$.  Let $\psi(X) = \prod_i (X_i, \codim X_i)$.  We claim that this defines a monoid map $L^{ch}(\calA) \to L^\mu(\calA)$ by extending multiplicatively.  

Let $a, b$ be hyperplanes in the same irreducible component $X_s$ of $X$.  Then 
\[\psi(a X) =  (X_s, \codim X_s + 1)\prod_{i\neq s} (X_i, \codim X_i)  = \psi(bX)\]
Thus, $\psi$ respects relation (2).  

Now suppose that $X, Y, Z \in L(\calA)$ with $X \cap Y = Z$ and $\codim Z = \codim X + \codim Y$.  As a warm up, consider first the case where $Z$ is irreducible.  Let $X_1, \ldots, X_n$ be the irreducible components of $X$ and $Y_1, \ldots, Y_m$ be the irreducible components of $Y$.  Since the irreducible components $X_i$ all intersect transversely and the irreducible components $Y_i$ intersect transversely, $\codim Z = \sum_i \codim X_i + \sum_j \codim Y_j$.  Thus
\[\begin{array}{rcl}
\psi(X)\psi(Y)  &=& \prod_{i=1}^n (X_i, \codim X_i) \times \prod_{j=1}^m (Y_j, \codim Y_j)\\
	        &=& (Z, \sum_i \codim X_i + \sum_j \codim Y_j)  \\
		&=& (Z, \codim Z) \\
                &=& \psi(Z)
\end{array}\]
The general case is not much different.  Let $Z_1, \ldots, Z_p$ be the irreducible components of $Z$, and let $I_q = \{i : Z_q \subset X_i \}$, $J_q =  \{i : Z_q \subset X_i \}$.  The $Z_i$ intersect transversely, and $Z_q = \left( \bigcap_{i \in I_q} X_i \right) \cap \left( \bigcap_{j \in J_q} Y_j \right)$, so a counting argument implies that $\codim Z_q =  \sum_{i \in I_q} \codim X_i +  \sum_{j \in J_q} \codim Y_j$.  Thus
\[\begin{array}{rcl}
\psi(X)\psi(Y)  &=& \prod_{i=1}^n (X_i, \codim X_i) \times \prod_{j=1}^m (Y_j, \codim Y_j)\\
	        &=& \prod_{q=1}^p \left( \prod_{i \in I_q} (X_i, \codim X_i) \times \prod_{j\in J_q} (Y_j, \codim Y_j) \right)\\
		&=& \prod_{q=1}^p \left( Z_q, \sum_{i \in I_q} \codim X_i +  \sum_{j \in J_q} \codim Y_j \right) \\
                &=& \prod_{q=1}^p \psi(Z_q, \codim Z_q) \\
                &=& \psi(Z)
\end{array}\]
Thus, $\psi$ respects relation (1).  We have shown that $\psi$ is well-defined, and it is clearly an inverse to $\phi$.
\end{proof}

\begin{definition}
Let $\fk{t}$ be a complex vector space.  A \emph{(complex) reflection} in
$GL(\fk{t})$ is an automorphism of finite order whose fixed locus is a
hyperplane.  A finite subgroup $W \subset GL(\fk{t})$ is a \emph{(complex)
reflection group} if it is generated by complex reflections \cite{shephard}.  There is a
hyperplane arrangement in $\fk{t}$ associated to $W$, consisting of the fixed
loci of all reflections in $W$.  The pair $(W, \fk{t})$ is called a
\emph{reflection arrangement}.  If $X \in L(\calA)$, write $W_X$ for the
subgroup of $W$ that fixes all points of $X$.
\end{definition}

\begin{lemma}
\label{irreps}
Let $(W, \fk{t})$ be a reflection arrangement with associated hyperplane arrangement $\calA$.  Suppose that $\calA$ is essential.  Then each irreducible representation of $W$ occurs in $\fk{t}$ with multiplicity at most $1$, and the decomposition of $\fk{t}$ into irreducible representations is identical to the decomposition of $\calA$ into irreducible hyperplane arrangements.  
\end{lemma}
\begin{proof}
We first show that if a reflection arrangement is irreducible as a hyperplane arrangement, then it is irreducible as a representation.  
Suppose that $\fk{t} = \fk{t}_1 \oplus \fk{t}_2$ as representations of $W$.  Every reflection in $W$ stabilizes $\fk{t}_1$ and $\fk{t}_2$, so every hyperplane in $\calA$ is pulled back from $\fk{t}_1$ or from $\fk{t}_2$.  Then $\calA = \calB_1 \oplus \calB_2$, where $\calB_i$ is the arrangement in $\fk{t}_i$ consisting of all hyperplanes in $\fk{t}_i$ that pull back to some hyperplane of $\calA$.  Further, neither of the $\calB_i$ are empty, since $\calA$ is essential.  This proves the sub-claim.

Since $\calA$ is essential, $\calA  =\calB_1 \oplus \ldots \oplus \calB_n$ where all the $\calB_i$ are irreducible and essential.  By the previous paragraph, the underlying vector space of $\calB_i$ is an irreducible representation of $W$.  Let $W_i \subset W$ be the subgroup generated by reflections over the hyperplanes in $\calB_i$.  Then $W = W_1 \times \ldots \times W_n$.  $\calB_i$ is fixed by $W_j$ for $j \neq i$, whereas $W_i$ acts on it nontrivially, so the $\calB_i$ are distinct representations of $W$.
\end{proof}

\begin{lemma}
\label{restrictToAx}
Let $(W, \fk{t})$ be a reflection arrangement, $\calA$ be the associated hyperplane arrangement and $X \in L(\calA)$.
\begin{enumerate}
\item $X = \fk{t}^{W_X}$.
\item $(W_X, \fk{t}/X)$ is an essential reflection arrangement.  
\item Let $\calA'$ be the hyperplane arrangement corresponding to $(W_X, \fk{t}/W)$.  Then $\calA_X = \calA' \oplus E$ for $E$ an empty arrangement.  In particular elements of $\irr \calA_X$ correspond bijectively to elements of $\irr \calA'$.
\end{enumerate}
\end{lemma}
\begin{proof}
$X \subset \fk{t}^{W_X}$ is tautological.  Since $X \in L(\calA)$, $X$ is the intersection of all hyperplanes in $\calA_X$.  In particular if $t \in \fk{t} - X$ then there is some hyperplane $a$ in $\calA_X$ that does not include $t$.  Let $w \in W$ be a reflection whose fixed locus is $a$.  Then $w \in W_X$ but $w$ does not fix $t$.  This shows (1).
Write $\calA_X = E + \bigoplus_{i=1}^n \calB_i$ where the $\calB_i$ are irreducible and essential.  Then $X \subset E$ and (1) implies that $X = E$.

$W_X$ acts naturally on $\fk{t}/X$.  \cite{Steinberg} theorem 1.5 implies that $W_X$ is generated by reflections over hyperplanes in $\calA_X$.  If $w \in W_X$ is reflection over a hyperplane $a \in \calA$, then $w$ acts on $\fk{t}/X$ by reflection over $a/X$.  This shows (2).  It also implies that $\calA'$ consists of the hyperplanes $a/X$ as $a$ ranges over $\calA_X$.  Since $X = E$, this implies that $\calA' = \bigoplus_{i=1}^n \calB_i$, which shows (3).
\end{proof}

\section{Cameral covers of $\spec \C$}\label{chapCameralCovers} 

In this section we define cameral covers, give some examples, and describe an induction operation on cameral covers.  We use this operation to classify cameral covers of $\spec \C$ (\cref{kPointsFinal}) and to compute their automorphism groups (\cref{pointStabalizers}).

\begin{definition}
Let $(W, \fk{t})$ be a reflection arrangement, and let $C \to S$ be an $S$-scheme with an action of $W$ that fixes the map $C \to S$.  $C$ is a $(W, \fk{t})$-\emph{cameral cover} of $S$ if there exists an \'{e}tale cover $U \to S$, a map $U \to \fk{t}/W$, and a $W$-equivariant isomorphism $U \times_{\fk{t}/W} \fk{t} \cong U \times_S C$ over $U$.  A \emph{pointed} cameral cover is a cameral cover with a chosen section.
\end{definition}

By the Chevalley-Shephard-Todd Theorem, $\fk{t}/W$ is the spectrum of a polynomial ring and the pushforward of $\calO_{\fk{t}}$ to $\fk{t}/W$ is locally free of rank 1 as a $\calO_{\fk{t}/W}[W]$-module (see \cite{lehrer2009unitary}, section 3.5).  Therefore if $f: C \to S$ is a cameral cover then $f_* \calO_C$ is locally isomorphic to $\calO_S[W]$.

\begin{example}
Here are some examples of cameral covers.
\begin{itemize}
\item  Any $W$-torsor is a cameral cover.
\item  Let $\calF$ be a locally free sheaf of rank $|W|$ on $S$, with a right $W$ action making it locally isomorphic to $\calO_S \otimes_\C \fk{t}$.  Then $\mathrm{Spec}_B\, \sym^\bul \calF^\vee /I$, where $I$ is the ideal sheaf generated by $W$-invariant sections of positive degree, is a cameral cover.  This cameral cover will be ramified everywhere.  The locally-defined map $S \to \fk{t}/W$ who witnesses that this is a cameral cover is the one factoring through the origin.
\item  Let $\fk{t} = \A^1$ and $W = \Sigma_2$ with the nontrivial element acting on $\A^1$ by $z \to -z$.  Let $L \to S$ be some line bundle.  Then the first infinitesimal neighborhood of the zero section of $L$, with $W$ action coming from the action of $W \cong \mu_2 \into \G_m$, is a cameral cover.  This is a special case of the previous example.
\item  Again let $(W, \fk{t}) = (\Sigma_2, \A^1)$.  In fact, any rank-$2$ finite morphism $C \to S$ has a canonical action of $W$ making it a cameral cover.  Here is the reason: Zariski locally on $S$, $C \cong \spec \calO_S[x]/(x^2 + bx + c)$ for some $b, c \in \calO_S$.  Completing the square, we may assume that $b = 0$.  Then the map $x \to -x$ defines an action of $W$, and this action does not depend on the choice of presentation.  $\fk{t} \to \fk{t}/W$ is isomorphic to the double cover $\spec k[u] \to \spec k[u^2]$.  Locally on $S$, $C \to S$ is isomorphic (over $S$, and as a $W$-scheme) to the pullback of $\fk{t} \to \fk{t}/W \cong \spec \C[u^2]$ along the map $S \to \fk{t}/W$ defined by $u^2 \to c$.
\end{itemize}
\end{example} 

Cameral covers generalize a better-known notion: 

\begin{definition}
A rank-$n$ \emph{spectral cover} $Y \to S$ is a rank-$n$ finite cover such that
there exists an \'{e}tale cover $X \to S$ and an embedding $Y \times_S X \into
\A^1 \times X$ over $X$. 
\end{definition}

Let $(W, \fk{t}) = (\Sigma_n, \C^{\oplus n})$, where $\Sigma_n$ acts by permuting the basis elements.   If $C \to S$ is a cameral cover, then $C/\Sigma_{n-1} \to S$ is a spectral cover.  This is one side of an equivalence of categories between rank-$n$ spectral covers and $(\Sigma_n, \C^{\oplus n})$-cameral covers (\cite{DG}, proposition 9.3)

\begin{definition}
Let $\pi:D \to S$ be an $S$-scheme with $W'$ action for some subgroup $W' \subseteq W$.  Define $\ind_{W'}^W D$ to be the balanced product $W \times_{W'} D$.  Equivalently $\ind_{W'}^W X = \pi_*\calO_D  \otimes_{\C[W']} \spec \C[W]$, where the $\C[W]$ factor is made a ring under pointwise multiplication.  
\end{definition}

\begin{lemma}
Let $\calA$ be the hyperplane arrangement in $\fk{t}$ associated to $(W, \fk{t})$,  $X \in L(\calA)$, $U = \fk{t}-\bigcup_{a \not\in \calA_X} a$ and $S$ be the image of $U$ in $\fk{t}/W_X$.  Let $C \to S$ be the restriction of the tautological $(W_X, \fk{t})$-cameral cover $\fk{t} \to \fk{t}/W_X$ to $U$.  $S$ maps to $\fk{t}/W$ via the composition $S \into \fk{t}/W_X \to \fk{t}/W$.  Using this map, $S \times_{\fk{t}/W} \fk{t} \cong \ind_{W_X}^W C$.
\label{lemUniversalInd}
\end{lemma}
\begin{proof}
Note that, by part 1 of \cref{restrictToAx}, $U$ includes precisely those elements of $\fk{t}$ which are not fixed by any $w \in W - W_X$.  There is a natural map $u: S \times_{\fk{t}/W_X} \fk{t} \to S \times_{\fk{t}/W} \fk{t}$.  $u$ is a map of $W_X$ schemes, so it induces a map $\bar{u}: W \times_{W_X} \left( S \times_{\fk{t}/W_X} \fk{t} \right) \to S \times_{\fk{t}/W} \fk{t}$.  We claim that this is an isomorphism, and since the domain and codomain are both normal (in fact smooth) varieties it suffices to show that $\bar{u}$ is a bijection on closed points.  Let $s \in S$ be a closed point and let $\bar{u}_s: W \times_{W_X} \left( s \times_{\fk{t}/W_X} \fk{t} \right) \to s \times_{\fk{t}/W} \fk{t}$ be the fiber of $\bar{u}$ over $s$.  Write $W_X \cdot s$ for the orbit of $s$. Restricting $\bar{u}_s$ to closed points we get the natural map $W \times_{W_X} (W_X\cdot s) \to W \cdot s$.  This map is a bijection since $s$ is not fixed by any element of $W-W_X$.
\end{proof}

\begin{proposition}
\label{corInductionWorks}
If $C \to S$ is a $(W_X, \fk{t})$ cameral cover then $\ind_{W_X}^W C$ is a $(W, \fk{t})$-cameral cover.  
\end{proposition} 
\begin{proof}
Being a cameral cover is preserved under pullback, and $C \to S$ is locally pulled back from $\fk{t} \to \fk{t}/W_X$.  Therefore it would suffice to show that for any $p \in \fk{t}/W_X$ there exists a neighborhood $V$ of $p$ and a map $g: V \to \fk{t}/W$ such that $\ind_{W_X}^W \fk{t}|_V$ is isomorphic to $g^*\fk{t}$.  Let $t_p \in \fk{t}$ mapping to $p$.  Choose $x \in X$ so that $t_p + x \not\in a $ for any $a \in \calA_X - \calA$.  Since $X$ is not contained in any such $a$, a generic choice of $x$ will accomplish this.  Let $V' \subset \fk{t}$ be the neighborhood of $t_p$ consisting of all points $t$ such that $t + x \not\in a$ for any $a \in \calA - \calA_X$, and let $V = V'/W_X$ be its image in $\fk{t}/W_X$.  This is a Zariski neighborhood of $p$, since finite quotient maps are open.  Let $f': V' \to \fk{t}$ be the map $t \to t+x$.  Since $x$ is $W_X$ invariant, $f'$ is $W_X$-equivariant.  Therefore it defines a map $f: V \to \fk{t}/W_X$, with image disjoint from the image of $X$.  Let $g: V \to \fk{t}/W$ be the composition of $f$ with the quotient map $q: \fk{t}/W_X \to \fk{t}/W$.  

\Cref{lemUniversalInd} shows that $q^*(\fk{t} \to \fk{t}/W)$ is isomorphic to $\ind_{W_X}^W(\fk{t} \to \fk{t}/W_X)$.  On the other hand, $f^* \fk{t}$ is isomorphic to $\fk{t}|_V$ since $f$ is an affine shift in a direction complementary to the action of $W_X$.
\end{proof}

\Cref{corInductionWorks} implies a pointed version of the same statement: if $(C \to S, \sigma: S \to C)$ is a pointed cameral cover, then $\ind_{W_X}^W C$ is a cameral cover with section $W_X \times \sigma: S \to W/W_X \times C$.
 
\begin{notation}
If $(H, \fk{t})$ is a reflection arrangement, let $Co(H, \fk{t})$ be the fiber of $\fk{t} \to \fk{t}/H$ over the origin.  
\end{notation}

$Co(H, \fk{t})$ is identical to the spectrum of the coinvariant algebra of $(H, \fk{t})$.  In other words, if we let $I \subset \sym^\bul \fk{t}^\vee$ be the ideal generated by $H$-invariant polynomials of strictly positive degree then $Co(H, \fk{t}) = \spec \left( \sym^\bul \fk{t}^\vee \right) /I$.  It is a cameral cover of $\spec \C$.  Note that $Co(H, \fk{t}) = Co(H, \fk{t}/\fk{t}^H)$.  $Co(H, \fk{t})$ is a local ring, and the tangent space at the maximal point is isomorphic as an $H$ representation to $\fk{t}/\fk{t}^H$.  

\begin{proposition}
\label{kPoints}
Let $C \to \spec \C$ be a cameral cover and $\sigma \in C$ a closed point.  
Let $H = \stab_W \sigma$, and write $T_\sigma C$ for the tangent space to $C$ at $\sigma$.  Note that this carries a canonical $H$-action. 
\begin{enumerate}
\item $H = W_X$ for some $X \in L(\calA)$.
\item $C \cong \ind_{W_X}^W Co(W_X, T_\sigma C)$.
\end{enumerate}
\end{proposition}
\begin{proof}
Fix an isomorphism between $C$ and $p \times_{\fk{t}/W} \fk{t}$ for some $\C$-point $p$ of $\fk{t}/W$.  Let $t \in \fk{t}$ mapping to $\sigma$, so that $H = \stab_W t$.  Let $X$ be the intersection of all hyperplanes including $t$.  Then \cite{Steinberg} theorem 1.5 shows that $H = W_X$.  

By definition $p \times_{\fk{t}/W_X} \fk{t} \cong Co(W_X, \fk{t})$.  By \Cref{lemUniversalInd}, $p \times_{\fk{t}/W} \fk{t} \cong \ind_{W_X}^W  Co(W_X, \fk{t})$.  $Co(W_X, \fk{t}) \cong Co(W_X, \fk{t}/\fk{t}^{W_X})$ and $T_\sigma Co(H, \fk{t})$ is isomorphic as an $W_X$ representation to $\fk{t}/\fk{t}^{W_X}$, so this implies (2).
\end{proof}

\begin{example}  
Let $\fk{t} = \A^1$ and $W = \Sigma_2$ acting on $\fk{t}$ by $z \to -z$.  Let $p \in \A^1$.  Consider the cameral cover $p \times_{\fk{t}/W} \fk{t}$.  Either $p$ is the origin or it isn't.  
\begin{itemize}
\item If $p$ is not the origin then the stabilizer of $p$ is $1 = W_{\fk{t}}$.  $Co(W_{\fk{t}}, \fk{t}) = \spec \C$, so $\ind_{W_{\fk{t}}}^W Co(W_{\fk{t}}, \fk{t}) = W$.  And indeed, direct computation shows $p \times_{\fk{t}/W} \fk{t} \cong W$.
\item If $p$ is the origin then the stabilizer of $p$ is $W_{0} = W$.  Thus $\ind_{W_0}^W Co(W_0, \fk{t}) = Co(W, \fk{t})$.  And indeed, $p \times_{\fk{t}/W} \fk{t} = Co(W, \fk{t})$ by definition.
\end{itemize}
\end{example}

\begin{corollary}
\label{kPointsFinal}
\hspace{2mm}
\begin{enumerate}
\item Isomorphism classes of pointed cameral covers of $\spec \C$  correspond bijectively to elements of $L(\calA)$.  
\item Isomorphism classes of cameral covers of $\spec \C$ correspond bijectively to $W$-orbits in $L(\calA)$.  
\end{enumerate}
\end{corollary}
\begin{proof}

The previous lemma shows that every cameral cover $C \to \spec \C$ is isomorphic to $\ind_{W_X}^W Co(W_X, \fk{t})$ for some $X \in L(\calA)$.  Let $C = \ind_{W_X}^W Co(W_X, \fk{t})$ and $C' = \ind_{W_Y}^W Co(W_Y, \fk{t})$.  Let $\sigma \in C, \sigma' \in C'$ be induced from the closed points of $Co(W_X, \fk{t})$ and $Co(W_Y, \fk{t})$ respectively.

For (1), it would suffice to show that $(C, \sigma) \cong (C', \sigma')$ as pointed cameral covers \iffw $X = Y$.  If they are isomorphic as pointed cameral covers then $W_X = \stab_W \sigma = \stab_W \sigma' = W_Y$, so $X = Y$ by part 1 of \cref{restrictToAx}.  The converse is clear.

Now for (2).  Let $f: C \to C'$ be an isomorphism, so $(C, \sigma) \cong (C, f(\sigma))$ as pointed cameral covers, and $\stab_W f(\sigma) = W_X$.  All closed points of $\ind_{W_Y}^W Co(W_Y, \fk{t})$ are in a single $W$-orbit, so there exists $g \in W$ with $g \cdot f(\sigma) = \sigma'$.  Then $W_Y = gW_Xg\oii$, but since $Y = \fk{t}^{W_Y}$ and $X = \fk{t}^{W_X}$ this implies $gX = Y$.  The converse is clear.
\end{proof}

\begin{example}  As in the previous example, let $\fk{t} = \A^1$ and $W = \Sigma_2$ acting on $\fk{t}$ by $z \to -z$.  The hyperplane arrangement consists of a single hyperplane $0$ (in this case a point on a line), so $L(\calA) = \{\fk{t}, 0 \}$.  Any cameral cover of a point is of the form $p \times_{\fk{t}/W} \fk{t}$ for some point $p \in \fk{t}$.  We saw in the previous example that there are two isomorphism classes of such covers, depending on whether $p =0$.  Which isomorphism class we get is determined entirely by the stabilizer group of $p$ in the pointed cameral cover $\left(p \times_{\fk{t}/W} \fk{t}, p \right)$, which is either $W_0$ or $W_{\fk{t}}$.  In this case there is an accidental isomorphism between $L(\calA)/W$ and $L(\calA)$.

\end{example}

We now consider the automorphism groups of cameral covers and pointed cameral covers of a point.  


\begin{notation}

For each $X \in L(\calA)$, choose a pointed cameral cover $C_X \to \spec \C$  with point $\sigma_0 \in C_X$ such that $\stab_W \sigma_0 = W_X$.  Write $\bar{Q}_X$ for its automorphism group as a cameral cover, and $Q_X$ for its automorphism group as a pointed cameral cover.  As explained in the previous paragraph, $\stab(X)$ acts on $Q_X$.  Whenever we write $Q_X \rtimes \stab(X)$ below we mean the semidirect product defined via this action.  Write $T_\sigma C_X$ for the tangent space to $C_X$ at $\sigma$.  This carries an action of $W_X$, and we write $\aut T_\sigma C_X$ for its group of automorphisms as a $W_X$-representation.  
\end{notation}

\begin{lemma}
With the notation of the previous paragraph, 
\label{pointStabalizers}
\begin{enumerate}
\item $\aut T_\sigma C_X \cong \G_m^{\# \irr \calA_X}$.
\item $Q_X = U \rtimes \aut T_\sigma C_X$ for a connected unipotent group $U$.
\item $\bar{Q}_X = Q_X \rtimes \stab(X)/W_X$.
\end{enumerate}
\end{lemma}
\begin{proof} \Cref{kPoints}, part 2 implies that $T_\sigma C_X \cong \fk{t}/X$
	as a $W_X$-representation.  \Cref{restrictToAx} says that $(W_X,
	\fk{t}/X)$ is an essential reflection arrangement with irreducible
	components corresponding to elements of $\irr \calA_X$.  \Cref{irreps}
	then implies that $\fk{t}/X$ decomposes as a $W_X$ representation into
	distinct irreducible representations corresponding to elements of $\irr
	\calA_X$.  (1) now follows from Schur's lemma.

	Let
$C_\sigma \cong Co(W_X, T_\sigma C_X)$ be the connected component of $\sigma$
in $C$.  Then $Q_X = \aut \, C_\sigma$, where $\aut \, C_\sigma$ means
automorphisms as an $(W_X, \fk{t})$-cameral cover, so for (2) we just need to
study $W_X$-equivariant automorphisms of $Co(W_X, T_\sigma C_X)$.  Let $R$ be
the coordinate ring of $Co(W_X, T_\sigma C_X)$, let $\fk{m}$ be its maximal
ideal, and let $R_n = R/\fk{m}^{n+1}$.  Note that $R_n=R$ for $n$ large enough.  $R$ is a quotient
of the power series ring $\C[ [T_\sigma C_X] ]$ by the ideal of non-invertible
$W_X$-invariant elements, and this ideal is preserved by any $W_X$-equivariant
automorphism of $\C [ [T_\sigma C_X] ]$, so any such automorphism induces an
automorphism of $R$.  Since any element of $\aut T_\sigma C_X$ induces
such an automorphism of $\C[ [T_\sigma C_X] ]$, this already shows that $Q_X \to \aut T_\sigma
C_X$ is a split surjection.

The degree-1
component of $R_n$ is the quotient of $T^\vee_\sigma C_X$ by its
$W_X$-invariant subspace, but the latter is trivial.  So the degree-1 component
of $R$ is $T^\vee_\sigma C_X$ itself.  Continuing to write $\mathrm{Aut}$,
$\mathrm{End}$ for
$W_X$-equivariant automorphisms and endomorphisms, we claim:
\begin{enumerate}[(i)]
	\item If $\psi \in \emo R_n$, $\psi\left(T_\sigma C_X \right) \subset
		\fk{m}R_n$.
	\item Restricting to $T^\vee_\sigma C_X$ defines a bijection 
		$\Gamma: \emo R_n \to \hom_{W_X}(T^\vee_\sigma C_X,
		\fk{m}R_n)$
	\item If $\psi \in \emo R_n$, then $\psi$ is invertible if and only if its
derivative $\fk{m}/\fk{m}^2 \to \fk{m}/\fk{m}^2$ is invertible. 
\end{enumerate}

For (i), note that $W_X$ acts trivially on $R_n/\fk{m}$ while $T_\sigma C_X$ has no trivial
component.  So any $W_X$-equivariant $\psi$ must send $T_\sigma C_X$ to
$\fk{m}R_n$.  For (ii), $R_n$ is generated in degree $1$ so $\Gamma$ is injective.  We show that it is
surjective: Any intertwiner
$\phi: T^\vee_\sigma C_X \to \fk{m}R_n$ lifts to an intertwiner $T^\vee_\sigma C_X \to
\C[T_\sigma C_X]$, inducing a ring endomorphism $\phi_1$ of $\C[T_\sigma C_X]$.
$\phi_1$  is $W_X$-equivariant, so it must fix the unique $W_X$-fixed point of
$\spec \C[T_\sigma C_X]$, namely the origin, and therefore induces an equivariant endomorphism
$\phi_2$ of the power series ring $\C[ [T_\sigma C_X] ]$.  Being equivariant, $\phi_2$ induces an endomorphism $R \to
R$, which induces $\psi: R_n \to R_n$ satisfying $\Gamma(\psi)=\phi$.  One direction of (iii) is obvious.  For the other direction, the
Jacobian criterion says that $\phi_2$ is
invertible if its total derivative at the origin of $\spec \C[
[T_\sigma C_X] ]$ is invertible.  If $n\geq 1$ then $\phi_2$ has the same derivative as $\psi$ at
the origin, while if $n=0$ then the statement is vacuous.  

Now we show (2). Let $\aut R_n $ be the group of $W_X$-equivariant
automorphisms of $R_n$, and $\emo R_n $ the monoid of equivariant
endomorphisms.  Fix $n\geq 2$.  There is a natural map $f: \emo R_n \to \emo
R_{n-1}$.
 Using $\Gamma$, $f$ can be identified with the map $\hom_{W_X}(T^\vee_\sigma
 C_X, R_n) \to \hom_{W_X}(T^\vee_\sigma C_X, R_{n-1})$ induced by $R_n \to
 R_{n-1}$.  In particular $f$ is surjective.  (ii) implies that an element $x
 \in \emo R_n$ lies in $\aut R_n $ if and only if $f(x)$ lies in
 $\aut R_{n-1} $.  So $\aut R_n \to \aut R_{n-1}$ is surjective and its kernel
 is identical to the kernel $K_n$ of $\emo R_n \to \emo R_{n-1}$.  But
 $\fk{m}^nR_n$ is square-zero in $R_n$, so $K_n$ is isomorphic to
the additive group of $W_X$-equivariant $R_n$-derivations from $R_n$ to
$\fk{m}^nR_n$.  Thus $Q_X$ is an iterated extension of $\aut T_\sigma C_X$ by
additive groups
\[Q_X = \aut R_j \to \ldots \to \aut R_3 \to \aut R_2 \to \aut T_\sigma C_X \]
so the kernel of $Q_X \to \aut T_\sigma C_X$ is unipotent.


Now for (3).  Let $\aut\, W/W_X$ denote the automorphism group of $W/W_X$ as a $W$-set.
There is a map $\bar{Q}_X \to \aut\, W/W_X$ by remembering only the action on
connected components of $C_X$.  We then have an exact sequence \[ 0 \to Q_X \to
\bar{Q}_X \to \aut\, W/W_X \] It is standard that $\aut\, W/W_X \cong
N(W_X)/W_X$, where $N(W_X)$ denotes the normalizer of $W_X$ in $W$.  If $w \in
W$ stabilizes $X$ then it normalizes $W_X$, and the converse is true as well
since $X$ is precisely the fixed locus of $W_X$.  Therefore $N(W_X) =
\stab_W(X)$, and $\aut\, W/W_X \cong \stab_W(X)/W_X$.   

We define a splitting $\rho: \stab(X)/W_X \to \bar{Q}_X$.  Recall that $C_X$ is
the balanced product $W \times_{W_X} Co(W_X, \fk{t}/X)$.  If $n \in
\stab(X)/W_X$ then $n$ acts naturally on $\fk{t}/X$.  This action does not
commute with the action of $W$ but it does preserve the invariant polynomials
on $\fk{t}/X$, and so induces an automorphism $\bar{n}$ of $Co(W_X, \fk{t}/W)$.
Further, if $g \in W_X$ then $ng(t + X) = g^{-n} n(t + X)$, so $\bar{n}\circ
L_g = L_{g^{-n}} \circ \bar{n}$.  Define $\rho(n): W \times_{W_X} Co(W_X,
\fk{t}/X) \to W \times_{W_X} Co(W_X, \fk{t}/X)$ by the formula $\rho(n)(w, c) =
(wn\oii, \bar{n}c)$ for any $w\in W$ and any point $c$ of $Co(W_X, \fk{t}/W_X)$.
This is well-defined since for any $g \in W_X$, \[\begin{array}{rcl}
\rho(n)(wg, c) &=& (wgn\oii, \bar{n}c) \\ &=& (wng^{n\oii}, \bar{n}c) \\ &=& (wn,
g^{n\oii} \bar{n}c) \\ &=& (wn, \bar{n}gc) \end{array}\]

To see that $\rho$ splits the map $\bar{Q}_X \to W/W_X$ it suffices to note
that it permutes the components of $C_X$ (corresponding to $W/W_X$) in a way
identical to the (right) action of $n$.  \end{proof}

\begin{remark}
\label{remFinitePart}
The conjugation action of $\stab(X)/W_X$ on $Q_X$ preserves the subgroup $\aut\, T_\sigma C \into Q_X$, and the action on $\aut\, T_\sigma C$ is induced from its natural action on $\fk{t}/X$.  Decompose $\calA_X = \calB_1 \oplus \ldots \oplus \calB_l \oplus E$ where the $\calB_i$ are the irreducible components of $\calA_X$.  $n \in \stab_W(X)$ must permute the $\calB_i$.  $\aut_W T_\sigma C$ is a product of $\G_m$'s corresponding to the $\calB_i$, and the conjugation action of $\stab_W(X)/W_X$ on $\aut_W T_\sigma C$ merely permutes these factors in the corresponding way.  
\end{remark}

\begin{example}  We return to our ongoing example.  Let $(W, \fk{t}) = (\Sigma_2, \A^1)$, so that $L(\calA) = \{\fk{t}, 0\}$.  Let $p \in \A^1$, and consider the cameral cover $C = p \times_{\fk{t}/W} \fk{t}$.  If $p = 0$ then $C \cong \spec \C[x]/(x^2)$ with automorphism group $\G_m$.  This is expected, since $\calA_0 = \calA$ has a single irreducible component.  If $p \neq 0$ then $C$ is the disjoint union of $2$ points, isomorphic to $W$ as a $W$-scheme.  Thus the automorphism group of $C$ is $\Sigma_2$.  This is expected, as $\calA_{\fk{t}}$ is empty and so has no irreducible components, while $\stab_W(0)/W_0 = W = \Sigma_2$.    
\end{example}

\section{The stack of cameral covers}\label{chapM} 

We fix a reflection arrangement $(W, \fk{t})$ and write ``cameral cover'' in place of ``$(W, \fk{t})$-cameral cover'' for the rest of this section.  Cameral covers form a stack on the \'{e}tale site of the category of $\C$-schemes (since finite covers with $W$-action form such a stack, and cameral covers form a subcategory defined by an \'{e}tale-local condition).  Write $\calM$ for this stack.  Write $\calC$ for the universal cameral cover on $\calM$.  Equivalently, $\calC$ is the stack of cameral covers with a chosen section.  Our goals for this section are to describe some nice covers of $\calM$ and $\calC$ as Artin stacks (\Cref{smoothCover}), define a stratification of $\calM$ and $\calC$ into classifying spaces (definition \ref{strataDef}), describe the normal bundles of the closed strata (\cref{corNormals}).  Our main arguments will use induction on this stratification.

Let $\calM_F$ be the open subscheme of the moduli space of commutative algebra structures on $\C[W]$ parametrizing cameral covers of $\spec \C$.  $\calM$ is  the stacky quotient of $\calM_F$ by $\aut_W \C[W]$, which is a product of general linear groups (see \cite{DG}, 2.7 and 2.8).  Similarly $\calC$ is the stacky quotient of the universal $W$-cover of $\calC_F$ by $\aut_W \C[W]$.  In particular $\calC$ and $\calM$ are Artin.  $\calM_F$ and $\calC_F$ are complicated so our first goal is to find simpler covers of $\calM$ and $\calC$.

\begin{proposition}
The cameral cover $\fk{t} \to \fk{t}/W$ induces a representable smooth surjection $\fk{t}/W \to \calM$.  
\label{smoothCover}
\end{proposition}
\begin{proof}
The definition of cameral cover implies immediately that $\pi: \fk{t}/W \to \calM$ is surjective.  Let $f: B \to \calM$ be a map from a scheme $B$ corresponding to some cameral cover $\pi: C \to B$.  Then $B \times_\calM \fk{t}/W$ is the functor sending a $B$-scheme $p: P \to B$ to the set of pairs $(g, \alpha)$ where $g: P \to \fk{t}/W$ is some morphism and $\alpha: p^*C \to g^*\fk{t}$ is an isomorphism of cameral covers.  

The idea of the following argument is simply that $(g, \alpha)$ is completely determined by the map $p^*C \to g^*\fk{t} \to \fk{t}$ and such maps are parametrized by a smooth space, but it will take us a lot of words to make this explicit.  Let $\shom_W(\fk{t}^* \otimes_\C \calO_B, \pi_* \calO_C)$ be the $W$-equivariant subsheaf of sheaf hom.  This is locally free (since $\pi_*\calO_C$ is a locally free $\calO_B[W]$-module), so its geometric realization $|\shom_W(\fk{t}^* \otimes_\C \calO_B, \pi_* \calO_C)|$ is smooth over $B$.  The subspace $U \subset |\shom_W(\fk{t}^* \otimes_\C \calO_B, \pi_* \calO_C)|$ parametrizing maps whose image generates $\pi_*\calO_C$ as an algebra is open - it is the complement of a determinantal locus involving the multiplication $\mu: \pi_*\calO_C \otimes \pi_*\calO_C \to \pi_*\calO_C$.  

We describe a map $B \times_\calM \fk{t}/W \to U$.  Suppose given $(g, \alpha) \in B \times_\calM \fk{t}/W(P)$ as above.  The composition of $\alpha$ with the natural map $g^*\fk{t} \to \fk{t}$ gives a $W$-equivariant morphism $p^*C \to \fk{t}$.  The latter corresponds to a $W$-equivariant map $\fk{t}^* \to \Gamma(p^* C, \calO_{p^*C}) = \Gamma(P, p^*\pi_*\calO_C)$, which is the same as a global section of $\shom_W(\fk{t}^* \otimes_\C \calO_P, p^*\pi_*\calO_C) = p^* \shom_W(\fk{t}^*\otimes_\C \calO_C, \pi_*\calO_C)$.  Such a global section defines a map $P \to |\shom_W(\fk{t}^* \otimes_\C \calO_C, \pi_*\calO_C)|$.  The condition that $\alpha$ was an isomorphism forces this to land in $U$.

We describe a map $U \to B \times_\calM \fk{t}/W$.  Suppose given a map $P \to U$ over $B$.  Running the equivalences of the previous paragraph in reverse, this corresponds to a $W$-equivariant map $h: p^*C \to \fk{t}$.  As $P$ is the affine quotient of $p^*C$ by $W$, $h$ induces a commuting map $g: P \to \fk{t}/W$.
\[\xymatrix{
p^*C \ar[r]^h \ar[d]& \fk{t} \ar[d]\\
B    \ar[r]^g       & \fk{t}/W
}\]
Now $h$ factors uniquely through some map $\alpha: p^*C \to g^* \fk{t}$.  Since
$h$ was obtained from a map to $U$ (rather than a map to
$|\shom_W(\fk{t}^*\otimes_\C \calO_B, \pi_*\calO_C)|$), $\alpha$ is a closed
immersion.  A closed immersion between two finite covers of the same degree
is an isomorphism, so $(g, \alpha)$ defines a section of $B \times_\calM \fk{t}/W$.

It is easy to check that these two maps are inverse; the key point is that the map $g$ induced in the second paragraph is unique given $h$.
\end{proof}

As  $\fk{t} = \fk{t}/W \times_{\calM} \calC$ we obtain a map $\fk{t} \to \calC$.  This map is merely $t \to t \otimes_{\fk{t}/W} \fk{t}$, with marked point $\sigma = (t, 0)$.  \Cref{smoothCover} implies

\begin{corollary}
\label{corNArtin} 
$\fk{t} \to \calC$ is a smooth representable surjection. 
\end{corollary}

\begin{example}
Let $(W, \fk{t}) = (\Sigma_2, \A^1)$.  Then $\calC$ is isomorphic to the stacky quotient $[\A^1/\G_m]$, where $\G_m$ acts with degree $1$.  For, suppose given a $\G_m$-torsor $E \to S$ and a $\G_m$-equivariant map $\phi: E \to \A^1$.  Let $\calJ$ be the line bundle associated to $E$, and $f$ be the linear function on $\calJ$ associated to $\phi$.  Then the locus on $\calJ^*$ cut out by $f^2 = 0$ is a cameral cover of $S$, with $\Sigma_2$ acting in the guise of $\mu_2$.  It has a section $\sigma$ as well, namely the dual to $f$.  This defines a functor $[\A^1/\G_m] \to \calC$.

Conversely, suppose given a cameral cover $C \to S$ with section $\sigma$.  $C$ is rank $2$, and $\calJ^* = \calO_C/\calO_S$ is a line bundle on $S$ into which $C$ embeds.  Indeed $\Sigma_2$ acts on $\calJ^*$ as well in the guise of $\mu_2$, and this embedding is $\mu_2$-equivariant.  $\sigma$ determines a section of $\calJ^*$, which defines a $\G_m$-equivariant map $\calJ \to \A^1$.  Restricting to complement of the zero section of $\calJ$ gives us a $\G_m$-torsor on $S$ with a $\G_m$-equivariant map to $\A^1$.  This defines a functor $\calC \to [\A^1/\G_m]$.

Since $\calM$ is the affine quotient $\calC/W$, we conclude that $\calM$ is the stacky quotient $[\A^1(2)/\G_m^2]$ (where the`` $(2)$'' indicates that $\G_m$ acts by degree $2$).  $\calC \to \calM$ is obtained by starting with the quotient map $\fk{t} \to \fk{t}/W$, otherwise known as $\spec \C[x] \to \spec \C[x^2]$, and taking the stacky quotient by the commuting action of $\G_m$.  

This example is a bit anomalous: In general it is not the case that $\calC \cong [\fk{t}/G]$ for some group acting on $\fk{t}$.  If it were, $G$ would need to fix the origin of $\fk{t}$, since every other point maps to a different isomorphism class in $\calM$.  But then $G = \aut Co(W, \fk{t})$, while $G$ would need to act simply transitively on $\fk{t} - \bigcup_{a \in \calA} a$.  We would then have $\aut Co(W, \fk{t}) \cong \fk{t} - \bigcup_{a \in \calA} a$ as schemes.  This is false already for $(W, \fk{t}) = (\Sigma_3, \C^{\oplus 2})$, where $\aut Co(W, \fk{t}) \cong (\G_a \rtimes \G_a) \rtimes \G_m$ while $\fk{t} \bs \bigcup_{a \in \calA}$ is the complement of three lines in the plane. 
\end{example}

\begin{corollary}
The natural maps $K^0\calM \to K_0\calM$ and $K^0\calC \to K_0\calC$ are isomorphisms.  In particular $K_0\calM$ and $K_0\calC$ are rings under (derived) tensor product.
\end{corollary}
\begin{proof}
$\calM$ and $\calC$ are smooth stacks.  Therefore $\calM_F$ and $\calC_F$ are
smooth, in particular reduced, in particular smooth varieties.  Proposition 2.20 of
\cite{merkurjev} implies that if $\mathcal{T}$ is the stacky quotient of a
smooth quasiprojective variety by an algebraic group, then the natural map $K^0\mathcal{T} \to K_0\mathcal{T}$ is an isomorphism.
\end{proof}

\begin{definition}
\label{defineStratification}
Let $L$ be a finite poset and $\calS$ be a stack.  A \emph{stratification of $\calS$ by $L$} is a collection of locally closed substacks $\calS_X \into \calS$ for all $X \in L$ such that 
\begin{enumerate}
\item If $X \neq Y$ then $\calS_X \cap \calS_Y = \emptyset$.
\item $X \geq Y$ \iffw $\calS_Y \subset \ol{\calS_X}$.
\item Every $\C$ point of $\calS$ factors through some $\calS_X$.
\end{enumerate}
If $P$ is an upwards-closed sub-poset of $L$ we write $\calS_P = \bigcup_{X \in P} \calS_X$.  We will abuse notation and write $\calO_X$ for $i_*\calO_{\ol{\calS}_X}$  where $i: \ol{\calS}_X \into \calS$ is the inclusion.
\end{definition}

\begin{definition}
\label{strataDef}
If $X \in L(\calA)$, let $\calC_X$ be the substack of $\calC$ consisting of cameral covers $C \to B$ with section $\sigma: B \to C$ such that $\stab_W \sigma = W_X$.  If $[X] \in L(\calA)/W$, let $\calM_{[X]}$ be the image of $\calC_X$ in $\calM$.   In other words, $\ol{\calM}_{[X]}$ is the stack of cameral covers that \'{e}tale-locally admit a $W_X$-invariant section.  
\end{definition}

Using \cref{kPoints} it is easy to see that the $\calC_X$ define a stratification of $\calC$ by $L(\calA)$, and that the $\calM_{[X]}$ define a stratification of $\calM$ by $L(\calA)/W$.  Wherever there is potential for confusion, $\calO_X$ will denote the sheaf on $\calC$, not the sheaf on $\calM$.

\begin{remark}
Note that $\calC_X \neq \calM_X \times_\calM \calC$.  In fact $\calM_X \times_{\calM} \calC$ is the stack quotient $[C_X/\bar{Q}_X]$. 
\end{remark}

\begin{example}
Let $(W, \fk{t}) = (\Sigma_3, \C^{\oplus 2})$.  Then $\calM$ has three strata, corresponding to cameral covers with fiber $\spec \C[\Sigma_3]$, $\spec \C[\Sigma_3/ \langle (12) \rangle] \otimes \C[r_1, r_2]/(r_1 + r_2, r_1 r_2)$ and $\spec \C[r_1, r_2, r_3]/(r_1+r_2+r_3, r_1r_2+r_2r_3+r_1r_3, r_1r_2r_3)$ respectively.  These strata are isomorphic to $B\Sigma_3$, $B\G_m$ and $B(\G_a \rtimes \G_m)$ respectively.
\end{example}

\begin{lemma}
\label{classifyingSpaces}
$\calM_{X}$ is isomorphic to the classifying space $B\bar{Q}_X$, and $\calC_X$ is isomorphic to the classifying space $BQ_X$.
\end{lemma}
\begin{proof}

We consider $\calC_X$ first.  $\calC_X$ has a single isomorphism class of $\spec \C$-points by \cref{kPoints}, so it would suffice to show that it is covered smoothly by a smooth variety.   This is true since its preimage in $\fk{t}$ is a disjoint union of open subsets of linear spaces.  $\calC_X$ surjects smoothly on $\calM_X$ since $X - \bigcup_{Y > X} Y \to \fk{t}/W$ is smooth onto its image.  Therefore $\calM_X$ is a classifying space as well.
\end{proof}

In particular there is an equivalence between coherent sheaves on $\calC_X$ and representations of $Q_X$, and similarly for coherent sheaves on $\calM_X$ and representations of $\bar{Q}_X$.  The last goal of this section is to find the representations associated to some natural sheaves on $\calC_X$.  Let $X \in L(\calA)$, let $i: \ol{\calC}_X \into \calC$ be the inclusion, and let $N_i$ be its normal bundle.  Since $\ol{\calC}_X$ is fixed by $W_X$, $i^*\Omega^\vee_{\calC/\calM}$ and $N_i$ are sheaves of $W_X$-representations.   
\begin{proposition}
\label{propRestrictions}
The natural map $\phi: i^*\Omega^\vee_{\calC/\calM} \to N_i$ is surjective, and its kernel is the sheaf of invariants $\left( i^*\Omega^\vee_{\calC/\calM} \right)^{W_X}$.
\end{proposition}
\begin{proof}
Write $\calF := \left(i^*\Omega^\vee_{\calC/\calM}\right)/\left( i^*\Omega^\vee_{\calC/\calM} \right)^{W_X}$.  Since $\phi$ is $W_X$-equivariant and $N_i$ has no nonzero $W_X$-invariants, $\phi$ factors through $\calF$.  To check that the resulting map $\calF \to N_i$ is an isomorphism it suffices to check that its pullback to $X \subset \fk{t}$ is an isomorphism.  This is a map of coherent sheaves on a variety so it suffices to check that its fiber over every closed point $y \in X$ is an isomorphism.  

Let $Y \in L(\calA)$ be maximal with $y \in Y$ (so in particular $Y \subset X$).  By \cref{kPoints}, the fiber of $\fk{t} \to \fk{t}/W$ through $y$ is isomorphic to $Co(W_Y, \fk{t}/Y)$, and in particular the fiber of $i^*\Omega^\vee_{\calC/\calM}$ is identified with $\fk{t}/Y$.  The pullback of $\phi$ to $y$ is the quotient map killing all vectors tangent to $X$, and since $X$ is the fixed locus of $W_X$ this is the same as killing the $W_X$-invariant subspace.  
\end{proof}

\begin{notation}
For each $a \in \calA$, let $I_a$ be the ideal sheaf corresponding to the closed immersion $\bar{\calC}_a \into \calC$. For $\calB$ an irreducible component of $\calA_X$, write $\pi_\calB$ for the map $\aut_{W_X}\, T_\sigma C_X \to \aut_{W_X} \,\calB = \G_m$, and let $\calO(d \calB)$ be the pullback along $\pi_\calB$ of the degree-$d$ character of $\G_m$.  
\end{notation}

\begin{corollary}
\label{corNormals}
\begin{enumerate}
\item Let $X \in L(\calA)$ and $i: \ol{\calC}_X \into \calC$ be the inclusion.  Then $N_i|_{\calC_X}$ is isomorphic to the cotangent space $T_\sigma C_X$ as a representation of $Q_X$.
\item Let $a \in \calA_Y$ and let $\calB$ be the irreducible component of $\calA_Y$ including $a$.  Then $I_a|_{\calC_Y}$ is isomorphic to $\calO(\calB)$ as a representation of $Q_X$.
\end{enumerate}
\end{corollary}
\begin{proof}
Consider the universal cameral cover with section, $\calC \times_\calM \calC \to \calC$.  Its section is the diagonal map $\Delta$.  \Cref{propRestrictions} together with the natural isomorphism $\Omega^*_{\calC/\calM} = N_{\Delta}$ implies that $\left( i^*N_\Delta\right)/\left( i^*N_\Delta \right)^{W_X} \cong N_i$.  Now we pull back along the map $\spec \C \to \ol{\calC}_X$ corresponding to the pointed cameral cover $(C_Y, \sigma)$, where $Y>X$.  The normal bundle of the universal section pulls back to the normal bundle of $\sigma \into C_Y$, i.e. $T_\sigma C_Y$.  Statement (1) follows from the case $X=Y$, and statement (2) follows from the case $X=a$.   

\end{proof}

\section{Characteristic classes of cameral covers}\label{chapCharCameral}

Let $\kappa$ be a ring.  In this section we compute the cohomology
ring $H^*(\calC, \kappa)$ (\cref{HN}).  If $\# W$ is invertible in $\kappa$ then we immediately obtain $H^*(\calM, \kappa)$ (\cref{HM}). 

The next few observations are more general than they need to be, but we feel that they make the argument cleaner.  Let $L$ be a finite poset and $\cal{S}$ be a regular Artin stack with a stratification by $L$ (see \cref{defineStratification} for what this means).  Suppose further that $H^*(\calS_X, \kappa)$ is concentrated in even degree for all $X \in L$.  Recall that $P(X)$ denotes the smallest upwards-closed subset of $L$ including $X$.

Let $P \subset L$ be an upwards-closed subset, $\mathcal{U} = \calS - \calS_P$, $X \in L - P$ be maximal, and $\mathcal{X} = \mathcal{U} \cup \calS_X$ so that $\mathcal{X}$ is open in $\calS$ and $\calS_X$ is closed in $\mathcal{X}$.  Let $i: \calS_X \into \mathcal{X}$ and $j: \mathcal{U} \into \mathcal{X}$ be the inclusions, and let $r$ be the codimension of $\ol{\calS}_X$ in $\calS$.  As $\mathcal{U}$ and $\calS_X$ are regular, there is an exact triangle
\[ Ri_*\ul{\kappa}_{\mathcal{S}_X}[-2r] \to \ul{\kappa}_{\mathcal{X}} \to j_*\ul{\kappa}_{\mathcal{U}} \to i_*{\kappa}_{\cal{S}_X}[-2r+1]\]
Because $H^*(\calS_X, \kappa)$ is concentrated in even degrees, the associated
long exact sequence gives a short exact sequence 
\[0 \to H^*(\calS_P,
\kappa[-2r]) \to H^*(\calS, \kappa) \to H^*(\mathcal{U}, \kappa) \to 0\]
The
first map is the gysin map and the second map is the restriction map.  If
$\kappa$ is a field then these short exact sequences imply \[ \dim H^i(\calS,
\kappa) = \sum_{X \in L} \dim H^{i-\codim \calS_X}(\calS_X, \kappa)\]

If $\calS_P$ is pure of codimension $r$ in $\calS$, Write $[\calS_P]$ for the image of $1 \in H^*(\calS_P, \kappa)$ under the gysin map $\gamma: H^*(\calS_P, \kappa) \to H^{*+2r}(\calS, \kappa)$ and call it the ``fundamental class'' of $\calS_P$.  If $i: \calS_P \into \calS$ is the inclusion and $\alpha \in H^*(\calS, \kappa)$, then the projection formula for cohomology gives $\gamma(i^*\alpha) = [\calS_P]\alpha$.  

\begin{lemma}
\label{lemSurjection}
Let $\calS$ a regular Artin stack with a stratification by $L$ such that $H^*(\calS_X, \kappa)$ is concentrated in even degrees for each $X \in L$.  Let $R \subset H^*(\calS, \kappa)$ be some subring, such that for all $X \in L$, the image of $R$ in $H^*(\calS_X, \kappa)$ is all of $H^*(\calS_X, \kappa)$.  Suppose further that $[\calS_{P(X)}] \in R$ for all $X \in L$.  Then $R = H^*(\calS, \kappa)$.
\end{lemma}
\begin{proof} Let $\alpha \in H^*(\calS, \kappa)$.  We use induction on the
	stratification to show that $\alpha \in R$.  Let $P$ be an
	upwards-closed sub-poset of $L$
	and assume inductively that there exists $\beta \in R$ such that
	$\alpha - \beta$ vanishes on the complement $\mathcal{U} = \calS -
	\calS_P$.  Let $X \in P$ be minimal.  We will show there exists $\beta'
	\in R$ with $\alpha - \beta'$ vanishing on $\mathcal{X} = \calS_X \cup
	\mathcal{U}$.  Let $i: \calS_X \into \mathcal{X}$ be the inclusion and
	$\gamma$ be the corresponding gysin map.  $\calS_X$ is of pure
	codimension in $\calS_P$, so by exactness of the gysin sequence,
	$\alpha - \beta = \gamma(\eta)$ for some $\eta \in H^*(\calS_X,
	\kappa)$.  By assumption, $\eta = i^*\eta'$ for some $\eta' \in R$.  By
	the projection formula, $\gamma(\eta)  = [\calC_X] \eta'$, but the
	latter is in $R$ as well since $[\calC_X] \in R$.  Let $\beta' = \beta
	+ [\calC_X] \eta'$.  \end{proof}

\begin{lemma}
\label{lemInjection}
Let $\calS$ a regular Artin stack with a stratification by $L$ such that $H^*(\calS_X, \kappa)$ is concentrated in even degree for each $X \in L$.  Suppose that all the natural maps $H^*(\calS, \kappa) \to H^*(\calS_X, \kappa)$ are surjective (as is the case when \cref{lemSurjection} applies).  
Assume that for each $X \in L$, $[\calS_{P(X)}] |_{\calS_X} \in H^*(\calS_X, \kappa)$ is not a zero divisor.  Then the map $\delta: \coprod_{X \in L} \calS_X \to \calS$ induces an injection $\delta^*: H^*(\calS, \kappa) \into H^*\left( \coprod_{X \in L} \calS_X, \kappa \right)$.
\end{lemma}
\begin{proof} 
Let $\alpha$ be an element of $\ker \delta^*$.  Let $P \subset L(\calA)$ be
some upwards-closed sub-poset, and let $\mathcal{U} = \calS \bs \calS_P$.  Assume inductively that $\alpha|_U = 0$.  Let $X \in P$ be minimal and let $\mathcal{X} = \calS_X \cup \mathcal{U}$.  We will show that $\alpha|_{\mathcal{X}} = 0$.  For the rest of the proof we abuse notation and write $\alpha$ for $\alpha|_{\mathcal{X}}$.

Let $i: \calS_X \into \mathcal{X}$ be the inclusion and $\gamma: H^{*}(\calS_X, \kappa) \into H^*(\mathcal{X}, \kappa)$ the gysin map.  By assumption $\alpha$ vanishes on $\mathcal{U}$, so $\alpha = \gamma(\beta)$ for some $\beta' \in H^*(\calS_X, \kappa)$.  By assumption $\beta' = i^*\beta$ for some $\beta \in H^*(\mathcal{X}, \kappa)$.  By the projection formula, $\alpha = [\calC_X] \beta$.  By assumption $i^*[\calS_X]$ is not a zero divisor.  Since $i^*\alpha = 0$ we must have $i^*\beta = 0$.  But then $\beta'=0$ so $\alpha = \gamma(\beta')$ was $0$. 

\end{proof}

We now apply these lemmas in the case $\calS = \calC$, $L = L(\calA)$.  Note that $\calC_X$ is the classifying space of an extension of a torus by a (smooth) unipotent group.  The classifying space of a unipotent group is contractible, so the cohomology of $\calC_X$ is that of a torus, and in particular is concentrated in even degrees.

\begin{proposition}
\label{propBetti}
Let $\Q[u_{\calB \in \irr \calA_X}]$ denote the polynomial ring on letters $u_\calB$ corresponding to the irreducible components of $\calA_X$.  $\stab(X)/W_X$ permutes the irreducible components, and therefore acts on $\Q[u_{\calB \in \irr \calA_X}]$ from the right.  Write $\Q[u_{\calB \in \irr \calA_X}]^{\stab(X)/W_X}$ for the invariant subring, and $\Q[u_{\calB \in \irr \calA_X}]^{\stab(X)/W_X}_d$ for its degree-$d$ component.  Then the Betti numbers of $\calM$ are given by the formula
\[ \dim H^i(\calM, \Q) = \sum_{[X] \in L(\calA)/W} \dim \Q[u_{\calB \in \irr \calA_X}]^{\stab(X)/W_X}_{i - 2 \codim X} \]
\end{proposition}
\begin{proof}
We let $L = L(\calA)/W$, where $[X] \leq [Y]$ if there exist representatives $X' \in [X]$, $Y' \in [Y]$ with $X' < Y'$, and apply the formula for the Betti numbers of a stratified stack given above:
\[ \dim H^i(\calM, \Q) = \sum_{[X] \in L(\calA)/W} \dim H^{i-\codim \calM_X}(\calM_X, \Q)\]
$\codim \calM_X = 2 \, \codim X$, so it suffices to show that $H^*(B\ol{Q}_X, \Q) = \Q[u_{\calB \in \irr \calA_X}]^{\stab(X)/W_X}$.  This follows from the Serre spectral sequence applied to the fibration
\[ \xymatrix{ 
BQ_X \ar[r] & B\bar{Q}_X \ar[d] \\
            & B \left( \stab(X)/W_X \right)
}\]
\end{proof}


\begin{notation}
For  $X \in L(\calA)$, we abuse notation and write $\calO_X$ for the
pushforward to $\calC$ of $\calO_{\ol{\calC}_X}$.  Let $c(X)$ be the Chern
class of $\calO_X$.  Denote by $I_X$ the ideal sheaf defining $\ol{\calC}_X$;
we will use this mostly when $X$ is a hyperplane $a \in \calA$.
\end{notation}

\begin{lemma}
\label{lemResolution}
Let $a_1, \ldots, a_r$ be minimal with $X = a_1 \cap \ldots \cap a_r$.  Then $\calO_X = \prod_{i=1}^r (1 - I_a) = \prod_a \calO_a$ in $K_0\calC$.
\end{lemma}
\begin{proof}
The relation $1 - I_a = \calO_a$ is tautological, so it suffices to check that $\tor^i(\calO_{a_p}, \calO_{a_q}) = 0$ for $i>0$.  The pullback of $\tor^i(\calO_{a_p}, \calO_{a_q})$ to $\fk{t}$ is $0$ since $a_i$ and $a_j$ are transverse hyperplanes.
\end{proof}

\begin{corollary}
$c(X)$ is concentrated in degree $2\, \codim  X$, and if $X, Y$ are transverse then $c(X)c(Y) = c(X \cap Y)$.
\end{corollary}
\begin{proof}
Let $a_1, \ldots, a_r$ be as in the previous lemma.  Note that $r = \codim X$ and that $c(a_i) = c \calO_{a_i} = -c_1 I_{a_i}$.   Thus $c(a_i)$ is concentrated in degree $2$.  \Cref{lemResolution} now implies that $c(X)$ is concentrated in degree $2r$.

\Cref{lemResolution} also implies that if $X$ and $Y$ are transverse then $\calO_X \cdot \calO_Y = \calO_{X \cap Y}$ in K-theory.  Since $c(X)$, $c(Y)$ and $c(X \cap Y)$ are each concentrated in a single degree, this implies $c(X)c(Y) = c(X \cap Y)$.
\end{proof}

\begin{lemma}
Every element of $H^*(\calC_X, \kappa)$ is a pullback from $H^*(\calC, \kappa)$ of a polynomial in the classes $c(a)$, $a \in \calA$.
\label{Hpullback}
\end{lemma}
\begin{proof}
Our identification of $\calC_X$ with the classifying space of an extension of
the torus $T_\sigma C_X$ by a unipotent group (\cref{pointStabalizers},
\cref{classifyingSpaces}) implies that every class of $H^*(\calC_X, \kappa)$ is
a polynomial in the first Chern classes of the representations $\calO(\calB)$.  By \cref{corNormals}, all such representations are obtained by pulling back certain of the $I_a$, so $H^*(\calC_X, \kappa)$ is generated by pullbacks of the $-c_1I_a = c(a)$.
\end{proof}

\begin{lemma} Let $X \in L(\calA)$.  Then $[\ol{\calC}_X] = \pm c_X$ \label{funClass} \end{lemma}

\begin{proof} It suffices to consider the case $\kappa = \Z$, since fundamental
	classes and chern classes descend to integer coefficients.  Let $r$ be
	the codimension of $X$.  The fundamental class $[\ol{\calC}_X] \in
	H^*(\calC, \Z)$ lives in degree $2r$.  In fact, since $\calC_X$ is
	connected, $[\ol{\calC}_X]$ is completely determined up to a unit as
	the $\Z$-generator of the classes in degree $2r$ that vanish on
	$\calC - \ol{\calC_X}$.  Since $c(X)$ is in the correct degree and
	vanishes on the complement of $\ol{\calC_X}$, it would suffice to show
	that it is not divisible by any non-unit integer.  This follows if
	$c(X)|_{\calC_X}$ is not divisible by any non-unit integer, which is
	what we now show.
  
Choose $\{a_1, \ldots, a_r\} \in \calA$ minimal with $a_1 \cap \ldots \cap a_r
= X$ as in \cref{lemResolution}.  Then $c(X) = \prod_{j=1}^r c(a_j)$.  Make the
identification $H^*(\calC_X, \Z) \cong \Z[u_i]_{\calA_X^i \in \irr
\calA_X}$.  Then $c(a_j)|_{\calC_X}$ is identified with $u_i$, where
$\calA_X^i$ is the irreducible component of $\calA_X$ including $a_j$.  In
particular $c(X)|_{\calC_X}$ is monic, and is not divisible by any non-unit
scalar.  \end{proof}

\begin{corollary}
\label{corHGeneration}
$H^*(\calC, \kappa)$ is generated by the $c(X)$.
\end{corollary}
\begin{proof}
Let $R$ be the subring generated by the $v_a$.  \Cref{funClass} and \cref{Hpullback} show that \cref{lemSurjection} applies.
\end{proof}

\begin{corollary}
\label{deltaInjectionC}
$\delta$ induces an injection $H^*(\calC_X, \kappa) \into H^*(\coprod_X \calC_X, \kappa)$.
\end{corollary}
\begin{proof}
The computation in the proof of \cref{funClass} shows that $[\ol{\calC}_X]|_{\calC_X}$ is nonzero for all $X$, hence not a zero divisor.  The previous corollary shows that the natural maps $H^*(\calC, \kappa) \to H^*(\calC_X, \kappa)$ are surjective for all $X$.  Now apply \cref{lemInjection}.
\end{proof}

\begin{theorem}
\label{HN}
The function $X \to c(X)$ extends to an isomorphism $\phi: \kappa[L^{\mu}(\calA)] \cong H^*(\calC, \kappa)$
\end{theorem}
\begin{proof}
First we must show that the map $\phi$ is well-defined.  We know already that if $X$ and $Y$ intersect transversely then $c(X)c(Y) = c(X \cap Y)$.  Suppose now that $a, b\in \calA$ with $a$ and $b$ in the same irreducible component of $\calA_X$.  We will show that $c(a)c(X) = c(b)c(X)$.  By the previous corollary, it would suffice to show that for all $Y$, $c(a)c(X)|_{\calC_Y} = c(b)c(X)|_{\calC_Y}$.  If $Y \not \geq X$ then $c(X)|_{\calC_Y} = 0$ and the equation holds trivially.  Otherwise, make the identification $H^*(\calC_Y, \kappa) \cong \kappa[u_i]_{\calA_Y^i \in \irr \calA_Y}$.  Since $a$ and $b$ are in the same irreducible component of $\calA_X$ they are also in the same irreducible component of $\calA_Y$, and so both pull back to $u_i$ for some $i$.  By \cref{descripMonoid2}, this shows that $\phi$ is well-defined.

\Cref{corHGeneration} implies that $\phi$ is a surjection, so it remains only
to show that $\phi$ is an injection.  Let $\alpha \in \kappa[L^{\mu}(\calA)]$
with $\phi(\alpha) = 0$.  Write $\alpha = \sum_{x \in L^\mu(\calA)} n_x x$ for
some $n_x \in \kappa$, all but finitely many of which are zero. For $x \in
L^\mu(\calA)$ write $|x|$ for the \emph{support} of $x$, that is, the
intersection of all $X_i \in L(\calA)$ appearing in an expression $x = X_1X_2
\ldots X_l$.  Assuming for contradiction that some $n_x$ is nonzero, let $Y \in
L(\calA)$ be minimal such that there exists $y \in L^{\mu}(\calA)$ with $|y| =
Y$ and $n_y \neq 0$.  In other words, $n_yy$ is a term of $\alpha$ with minimal
support among all terms of $\alpha$, and that support is $Y$.  By minimality of
$Y$, if $n_xx$ is a nonzero term of $\alpha$ with $|x| \neq Y$ then $|x|
\not\leq Y$ and $x_{\calC_Y} = 0$.  Therefore $\phi(\alpha)|_{\calC_Y} =
\sum_{|y| = Y} n_y\phi(y)$.  

Now consider all $y \in L^\mu(\calA)$ with $|y| = Y$.  Their images under
$\phi$ in $H^*(\calC_Y, \kappa) \cong \kappa[ u_1, \ldots, u_l]$ (where $l$ is
the number of irreducible components of $\calA_Y$) are independent over
$\kappa$.  In particular, $\phi(\alpha)|_Y =0$ implies that $n_y = 0$
for all $|y| = Y$, a contradiction.  \end{proof}

$W$ has a natural right action on $H^*(\calC, \kappa)$, which corresponds under $\phi$ to the right action on $L^\mu(\calA)$ induced by sending $X$ to $w\oii X$.  Since $\calC \to \calM$ is covered by the finite quotient map $\fk{t} \to \fk{w}/W$, transfer implies
\begin{theorem}
\label{HM}
If $\#W$ is invertible in $\kappa$ then $H^*(\calM, \kappa) = \kappa[L^\mu(\calA)]^W$ as graded rings. 
\end{theorem}

\begin{remark} 
All arguments before \cref{HM} still work verbatim if $\#W$ is not a unit in
$\kappa$.  If $\fk{t}$ is one-dimensional then $\calM$ is
homotopy-equivalent to $B\G_m$ and the isomorphism of \cref{HM} works
integrally.  In general, computing $H^*(\calM, \kappa)$ together with its
filtration would be at least as hard as computing the group cohomology of $W$.
\end{remark} 

\section{Characteristic classes of induced cameral covers}\label{chapInduced} 

Fix a reflection arrangement $(W, \fk{t})$ and let $\calA$ be the corresponding hyperplane arrangement in $\fk{t}$.  If $X \in L(\calA)$, write $\calM^X$ for the stack of $(W_X, \fk{t})$ cameral covers, and $\calC^X$ for the associated stack of pointed cameral covers.  The induction functor gives a map $\calM^X \to \calM$.  If $(C, \sigma)$ is a pointed $(W_X, \fk{t})$-cameral cover then $(W_X \times \sigma) \in W/W_X \times C$ is a point of $\ind_{W_X}^W C$, so $\calM^X \to \calM$ lifts to a map $\calC^X \to \calC$ (indeed there is a lift for each element of $W$).  In this section we describe the maps of cohomology rings induced by these maps (\cref{refCohomology}).  We also define a polynomial invariant $p_c(r) \in H^*(S, \kappa)[r]$ of a cameral cover $C \to S$, and use \cref{refCohomology} to show that it is multiplicative under disjoint unions of spectral covers (\cref{whitney}).

\begin{proposition}
\label{refCohomology}
Let $P = L(\calA) \bs L(\calA_X)$.
\begin{enumerate}
	\item $P$ is upwards-closed.  
\item The map $H^*(\calC, \kappa) \to H^*(\calC^X, \kappa)$ induced by induction is a surjection with kernel generated by $P$.  
\item If $\# W$ is invertible in $\kappa$ then the map $H^*(\calM, \kappa) \to H^*(\calM^X, \kappa)$ is a surjection with kernel generated by $P \cap \kappa [L^\mu(\calA )]^W$.
\end{enumerate} 
\end{proposition}
\begin{proof}
$P$ can be characterized as the set of $Y \in L(\calA)$ such that $Y \cap X$ is
strictly contained in $X$, so it is upwards-closed.  Let $\calU = \calM - \calM_P$ and $\ol{\calU} = \calC - \calC_P$.  If $(C, \sigma)$ is a pointed $(W_X, \fk{t})$-cameral cover then the section $1 \times \sigma$ of $\ind_{W_X}^W C$ is never fixed by $W_Y$ for $Y \in P$, so $\calC^X \to \calC$ factors through $\ol{\calU}$.  Since $\calC^X$ surjects to $\calM^X$ this implies $\calM^X$ factors through $\calU$.  We therefore have a commutative diagram as below, where the vertical maps are the forgetful functors.  (Caution: the squares are not Cartesian.)
\[ \xymatrix{
\calC^X \ar[r] \ar[d] & \ol{\calU} \ar[r] \ar[d] & \calC \ar[d] \\
\calM^X \ar[r]        & \calU \ar[r]              & \calM }\]
In fact the map $\calC^X \to \ol{\calU}$ is an isomorphism.  It is clearly faithful.  It is also full - this is the statement that if $(C, \sigma)$ and $(C', \sigma')$ are $(C_X, \fk{t})$-cameral covers then any map $\ind_{W_X}^W C \to \ind_{W_X}^W C'$ sending the induced section $W_x \times \sigma$ to $W_X \times \sigma'$ is induced from its restriction to the connected components of the sections $(C, \sigma) \to (C', \sigma')$.  It remains only to show that $\calC^X \to \ol{\calU}$ is essentially surjective.  As both stacks are locally finite type and reduced it suffices to check that the map is surjective on $\C$-points, but this follows from \cref{kPoints}.  

In particular, $H^*(\calC, \kappa) \to H^*(\calC^X, \kappa)$ is a surjection with kernel generated by the fundamental classes of the closed strata of $\calC_P$, so the kernel is generated by $P$.  The claim for $H^*(\calM, \kappa) \to H^*(\calM^X, \kappa)$ follows from commutativity of the diagram.
\end{proof}

Now suppose that $X \in L(\calA)$ and suppose that $\calA_X$ decomposes as $\calA_X = \calA_Y \oplus \calA_Z$.  Let $\fk{t} = \fk{t}_Y \oplus \fk{t}_Z$ be the corresponding splitting.  Given a $(W_Y, \fk{t}_Y)$-cameral cover $C_Y \to S$ and a $(W_Z, \fk{t}_Z)$-cameral cover $C_Z \to S$, $C_Y \coprod C_Z$ has naturally the structure of a $(W_Y \times W_Z, \fk{t})$-cameral cover.  Note that $W_Y \times W_Z = W_X$.  Denote by $C_Y \oplus C_Z$ the $(W, \fk{t})$ cameral cover $\ind_{W_X}^W C_Y \coprod C_Z$.

\begin{example}
Let $(W, \fk{t}) = (\Sigma_{n+m}, \C^{n \oplus m})$.  Let $Y$ be the subspace of $\fk{t}$ where the first $n$ coordinates are identical and let $Z$ be the subspace where the last $m$ coordinates are identical.  Let $X = Y \cap Z$.  Then $C_Y \oplus C_Z$ makes sense whenever $C_Y$ and $C_Z$ are cameral covers for $\Sigma_n$ and $\Sigma_m$ respectively.  In this situation we can let $S_Y$ and $S_Z$ be the corresponding spectral covers.  $C_Y \oplus C_Z$ corresponds to the spectral cover $S_Y \coprod S_Z$.
\end{example}

We now mimic the Whitney sum formula for the Chern polynomial of a vector bundle.  Let $p(r) \in H^*(\calM, \kappa)[r]$ be the polynomial $\prod_{a \in \calA} (1 + ar)$.  For a cameral cover $C \to X$, let $p_C(r) \in H^*(X, \kappa)[r]$ be the image of $p(r)$.  Note that, unlike the chern polynomial of a vector bundle, $p_C$ does \emph{not} encode all characteristic classes of $C$.

\begin{proposition}
\label{whitney}
$p_{C_Y \oplus C_Z}(r) = p_{C_Y}(r)p_{C_Z}(r)$
\end{proposition}
\begin{proof}
It suffices to consider the global case, where $C_Y$ and $C_Z$ are the pullbacks to $\calM^Y \times \calM^Z$ of $\calC^Y$ and $\calC^Z$ respectively.  Then $p_{C_Y \oplus C_Z}(r)$ is the image in $H^*(\calM^Y, \kappa) \otimes H^*(\calM^Z, \kappa)[r]  = H^*(\calM^Y \times \calM^Z, \kappa)[r]$ of $\prod_{a \in \calA} (1  + ar)$ under the natural map from $H^*(\calM, \kappa)$.  This map merely kills all $a$ that are not in $\calA_X$.  Since $\calA_X = \calA_Y \oplus \calA_Z$, the image is $\prod_{a \in \calA_Y} (1 + ar) \times \prod_{a \in \calA_Z} (1 + ar) = p_{C_Y}(r)p_{C_Z}(r)$.
\end{proof}

\section{Characteristic classes of Higgs bundles}\label{chapCharHiggs} 

Let $G$ be a connected affine reductive group over $\C$.  In this section we use the results of \cite{DG} to compute the rational cohomology ring of the stack of $G$-Higgs bundles as a $H^*(\calM, \Q)$-algebra.  Since we are only interested in rational coefficients for this section we write $H^*(-)$ in place of $H^*(-, \Q)$.

Let $T$ be a maximal torus of $G$, $N$ be its normalizer, and $B$ be a Borel subgroup of $G$ containing $T$.  Let $\fk{g}$, $\fk{t}$ be the Lie algebras of $G$ and $T$.  Let $W$ be the Weyl group of $G$, acting by reflections on $\fk{t}$ in the usual way.  Choose a system of positive roots $\Phi^+ \in \fk{t}^*$ and let $\Phi^-$ be the negative roots.  For each $X \in L(\calA)$ let $N_X$ be the preimage of $W_X$ under the projection $N \to W$.  We begin with some definitions, mostly from \cite{DG}.

\begin{definition}
A \emph{regular centralizer} in $\fk{g}$ is the centralizer in $\fk{g}$ of a regular element.  We do not assume that the regular element is semisimple. 

A \emph{(abstract regular) Higgs bundle} on $S$ is a right 
 $G$-torsor $E \to S$ together with a sheaf of regular centralizers $\fk{a}$ inside the sheaf of Lie algebras $\fk{e}$ associated to $E$.  A morphism of Higgs bundles is a morphism of principal bundles that preserves the chosen sheaves of regular centralizers. Write $\calH$ for the stack of Higgs bundles, i.e. $\calH(S)$ is the category of Higgs bundles on $S$.  $\calH$ is an Artin quotient stack (\cite{DG}, paragraph 2.3).  

Given a Higgs bundle $(E, \fk{a})$ on $S$ we can form the moduli space of Borel algebras in $\fk{e}$ containing $\fk{a}$.  This space carries a natural left action of $W$ which makes it a cameral cover of $S$.  This defines a functor $h: \calH \to \calM$ called the \emph{Hitchin map}.  

A \emph{pointed Higgs bundle} is a Higgs bundle $H$ together with a section of $h(H)$.  In other words, a pointed Higgs bundle is a triple $(E, \fk{a}, \fk{b})$ where $E$ is a principal $G$-bundle with sheaf of Lie algebras $\fk{e}$, and $\fk{a} \subset \fk{b} \subset \fk{e}$ is a chain of sheaves of Lie algebras where $\fk{a}$ is a regular centralizer and $\fk{b}$ is a Borel subalgebra.  Write $\calH_\calC = \calH \times_\calM \calC$ for the stack of pointed Higgs bundles, and $\calH_X = \calH_\calC \times_\calC \calC_X$.  This gives a smooth lci stratification of $\calH_\calC$ by $L(\calA)$.  
\end{definition}

\begin{example}
\label{exHiggs1}

Let $G = GL_n$.  Suppose given a locally free sheaf $\calF$ on $S$ and a map $x: \calF \to \calF \otimes \calJ$ where $\calJ$ is some line bundle.  Such a pair is what most of the literature calls a Higgs bundle.  Identify $x$ with a section of $\calJ \otimes \emo \calF$.  Let $U \subset S$ be open, and suppose there is an isomorphism $\alpha': \calJ|_U \to \calO_U$.  This induces an isomorphism $\alpha: \calJ \otimes \emo \calF|_U \to \emo \calF|_U$.  Assume that $\alpha(x)$ is a regular element.  Note that this does not depend on the choice of trivialization of $\calJ$.  Let $\fk{a}_{U, \alpha} \subset \emo \calF|_U$ be the centralizer of $\alpha(x)$.  $\fk{a}$ does not depend on the choice of $\alpha'$, so this defines a sheaf of regular centralizers $\fk{a} \subset \emo \calF$.  Let $E$ be the $GL_n$-principal bundle associated to $\calF$, so that $\fk{e} = \emo \calF$.  We have constructed a Higgs bundle $(E, \fk{a})$.  

Suppose further that $x$ is semisimple.  Choosing a Borel subalgebra containing the centralizer of $x$ is the same as ordering the eigenspaces of $x$, so $h(E, \fk{a})$ is the space of ordered eigenspaces of $x$.  The action of $W = \Sigma_n$ permutes the eigenspaces.
\end{example}

Our goal is to show that a map $q: \calH_\calC \to \calC \times BT$, originally defined in \cite{DG}, induces a surjection on rational cohomology and to compute the kernel.  We will identify the images of the strata $\calH_X$ under $q$, and reduce the problem to checking that $q$ induces a surjection on cohomology of strata.  The strata will turn out to be classifying spaces of groups of the form (unipotent group) $\rtimes$ (torus) $\rtimes$ (finite group), and the problem reduces further to studying the associated map between tori.  Finally, we identify the $W$-action on $H^*(\calH_\calC, \Q)$ in order to express $H^*(\calH, \Q)$ as the subring of $W$-invariants.

\cite{DG} completely describes $\calH$ as a gerbe over $\calM$.  Insofar as our description of $\calH$ looks different from the one in \cite{DG} it is because we are only interested in cohomology and can therefore afford to ignore some information.  Here are some auxiliary constructions from \cite{DG} that we will need:

\begin{itemize}
\item If $(E, \fk{a}, \fk{b})$ is a pointed Higgs bundle on $S$, $\fk{b}$ defines a lift of $E$ to a $B$-bundle $E_B$, which defines an $E_T$ bundle on $S$ using the projection $B \to T$.  This defines a map $\calH_{\calC} \to BT$, hence a map $q: \calH_\calC \to \calC \times BT$.  Let $\calL$ be the corresponding $T$-bundle on $\calH_\calC$.    

\item  Let $\map_{W}(\calC, T)$ be the abelian algebraic $\calM$-group
	associating to a cameral cover $C \to S$ the space of $W$-equivariant
	maps $C \to T \times S$ over $S$.  For any root $\alpha$ of $G$ there
	is a map $\alpha': T \to \G_m$ and an associated hyperplane $a_\alpha \in \calA$.  Let $\calT$ be the subsheaf of $\map_{W}(\calC, T)$ consisting maps $C \to T$ such that, under the composition $C \to T \to \G_m$, $a_\alpha$ is sent to $1$.  Since $a_\alpha$ always goes to $\pm 1$, $\calT$ has finite index in $\map_{W}(\calC, T)$.

\item For any root $\alpha$ of $G$, let $a_\alpha \in \calA$ be the corresponding hyperplane.  Recall that $I_a$ is the corresponding ideal sheaf on $\calC$.  Let $\calR^{\alpha}$ be the $T$-bundle of $\calC$ obtained by inducing $I_{a_\alpha}^\vee$ (viewed as a $\G_m$ torsor) up to a $T$-bundle $\calR^{\alpha}$ via the dual coroot $\check{\alpha}: \G_m \to T$.  For $w \in W$, let 
\[\calR^w = \bigotimes_{\substack{\alpha \in \Phi^+ \\ w(\alpha) \in \Phi^-}} \calR^\alpha\]
and for $n \in N$ let $\calR^n = \calR^w$ where $w$ is the image of $n$ in $w$.  
\end{itemize}

Given a $T$-torsor $E_T \to S$ and $w \in W$, write $E_T^w$ for its twist by
$W$.  This is an action of $W$ on $BT$.  $W$ acts on $\calC$ as well, so it
acts on $\calC \times BT$ via the diagonal action.  $N$ maps to $W$, so it acts
on $\calC \times BT$ as well.

\begin{theorem}  (Donagi-Gaitsgory) \label{thmDGfacts} 
	\begin{enumerate} \item
			(\cite{DG}, theorem 4.4) $h: \calH \to \calM$ is a
			gerbe with band $\calT$.  
		\item (\cite{DG}, see
			\cref{DGexplainer})  The map $q: \calH_\calC \to \calC
			\times BT$ is compatible with $\calT$, in the following
			sense:  Let $H \in \calH_\calC(S)$ be a pointed cameral
			cover over $S$ with associated pointed cameral cover
			$(C, \sigma)$, and let $x$ be an automorphism of $H$
			inducing the identity of $(C, \sigma)$.  By (1), $x$
			corresponds to a $W$-equivariant map $\phi: C \to T$.
			Then $q$ sends the automorphism $x$ to the automorphism
			of $\calL|_S$ defined by acting with $\phi(\sigma) \in
			\Gamma(S, T)$.   
		\item (\cite{DG}, lemma 5.4) $\calR$
			is a cocycle for the action of $W$ on $\calC \times
			BT$, in the sense that there is a canonical isomorphism
			$\calR^{w_1 w_2} \to w_2^*\left(\calR^{w_1}
			\right)^{w_2} \otimes \calR^{w_2}$.  
		\item (\cite{DG},
			theorem 6.4) $\calL$ is $\calR$-twisted
			$N$-equivariant, in the sense that for each $n \in N$
			mapping to $w \in W$ there is an isomorphism $\alpha_n:
			w^*\calL^w \otimes \calR^w \cong \calL$ such that for
			$t\in T$, $\alpha_t: \calL \to \calL$ is scaling by
			$t$, and the maps $\alpha_n$ are compatible in the
			natural way using the isomorphism of (3).  
	\end{enumerate}
\end{theorem}

Warning: If $E_T$ is a $T$-bundle on $\calC$, the sheaf $w^*E_T$ of \cite{DG}
is our $w^*E_T^w$.

\begin{remark}
\label{DGexplainer} 
We explain \cref{thmDGfacts} (1) more concretely.  Let $H, C, \sigma, x$ be as in the statement of part (2).  Part (1) implies that $x$ corresponds to an element $\phi \in \map_{W}(C, T)$.  The construction is as follows: Write $H = (E, \fk{a}, \fk{b})$.  For every point of $C$, corresponding to a Borel subalgebra $\fk{b}' \subset \fk{e}$ containing $\fk{a}$, we get a lift of $E$ to some $B$-torsor $E_B(\fk{b}')$.  Since $x$ induces the identity of $C$ it must stabilize $\fk{b}'$, and therefore lifts to an automorphism $x_B$ of $E_B(\fk{b}')$, which maps to an automorphism $x_T \in \Gamma(S, S \times T)$ of the $T$-torsor $T \times_B E_B(\fk{b}')$.  Thus $x$ determines a function from points $\fk{b}'$ of $C$ to $\Gamma(S, T)$, hence a function $\phi\in \map(C, T)$, which turns out to land in $\calT \subset \map(C, T)$.  This description is taken from paragraphs 11.5, 11.9 and 14.3 of \cite{DG}.

We can now justify (2): $\phi(\sigma)$ is the element of $\Gamma(S, T)$ that we obtained when $\fk{b}'=\fk{b}$.  In this case $T \times_B E_B = \calL|_S$, so $q(x) = x_T$ by definition of $q$.  

\end{remark}


\begin{remark}
$\calL$ is not twisted $W$-equivariant: the isomorphisms of \cref{thmDGfacts} (4) depend on the lift $n$ of $w$.
\end{remark}

From now on we will only consider the $\calR$-twisted action of $N$ on $\calC \times BT$.  Concretely, if $(C, \sigma)$ is a pointed cameral cover of $S$ corresponding to a map $f: S \to \calC$, and if $E$ is a $T$-torsor on $S$, then $n \cdot (C, \sigma, E) \in \left(\calC \times BT\right)(S)$ is $(C, w\cdot \sigma, w^*E^w \otimes f^*\calR^w)$.  The point is that under this action the map $q: \calH_\calC \to \calC \times BT$ is $N$-equivariant.  

\begin{definition}
For $X \in L(\calA)$, let $(\calC_X \times BT)^{N_X}$ be the space of fixed
points $\calC_X \times BT$ under the (twisted) action of $N_X$.  Explicitly,
$(\calC_X \times BT)^{N_X}$ sends a scheme $S$ to the category of pairs $(C,
\sigma, \calJ)$ where $(C, \sigma)$ is a pointed cameral cover in $\calC_X(S)$
and $\calJ$ is a $\calR|_S$-twisted $N_X$-equivariant $T$-torsor. 
\end{definition}

Twisted equivariance of $q$ implies that $q_X := q|_{\calH_X}: \calH_X \to \calC_X \times BT$ factors through $(\calC_X \times BT)^{N_X}$.  Our strategy for the rest of this section is to study the map that $q$ induces on cohomology by studying these restrictions.  $\calH_X$ is the classifying space of an extension of $Q_X$ by $\calT_X \subset \map_W(C_X, T)$, the fiber of $\calT$ over $C_X \in \calM$.  Meanwhile $(\calC_X \times BT)^{N_X}$ is the classifying space of an extension of $Q_X$ by $T^W$.  Therefore $q_X$ is induced by a map of groups.  In fact it will suffice to study the induced map $\calT_X \to T^W$.  

%

\begin{lemma}
\label{isoFibers}
The map $\calT_X \to T^{W_X}$ defined by evaluation at $\sigma$ induces an isomorphism $H^*(BT^{W_X}) \to H^*(B\calT_X)$.
\end{lemma}
\begin{proof}

$\calT_X$ involves a lot of unipotent cruft that we would like to get rid of.  Recall that we fixed a representative $(C_X, \sigma) \in \calC_X(\spec \C)$. Let $C_X^0$ be the reduced subscheme of $C_X$, and note that there is a retraction $C_X \to C_X^0$, inducing $\map_W(C_X^0, T) \to \map_W(C_X, T)$.  We have a commutative diagram
\[\xymatrix{
\calT_X \ar[rr]^c && \map_W(C_X, T)          \ar[r]          &  T^{W_X} \\
                 & & \map_W(C_X^0, T) \ar[u]^b \ar[ur]_a       &
}\]
and it would suffice to show that $a$, $b$, and $c$ each induces an isomorphism on $H^*(B(-))$.

\textbf{1.  $a$ is an isomorphism:}  It is an injection because $W$ acts transitively on $C_X^0$, and a surjection since one can define a map $C^0 \to T$ by selecting any $W_X$-invariant point of $T$ to be the target of $\sigma \in C_X$.

\textbf{2.  $b$ induces a homotopy equivalence $B\map_W(C^0_X, T) \to B\map_W(C_X, T)$:}  The cokernel of $\map_W(C^0_X, T) \to \map_W(C_X, T)$ is the group $K$ of $W$-equivariant maps sending $\sigma$ to $1 \in T$.  Indeed, $K$ is a subgroup of $\map_W(C_X, T)$ which intersects trivially with $\map_W(C^0_X, T)$ and the two of them generate $\map_W(C_X, T)$.  Standard deformation theory implies that $K$ is contractible, so $BK$ is contractible.  On the other hand $B\map_W(C_X, T) \to BK$ is a fibration with fiber $B\map_W(C_X^0, T)$, so $B\map_W(C^0_X, T) \to B\map_W(C_X, T)$ is a homotopy equivalence.

\textbf{3.  $c$ induces an isomorphism $H^*(B\map_W(C_X, T)) \to H^*(B\calT_X))$:}  The cokernel of $\calT_X \to \map_W(C_X, T)$ is some finite group $F$.  There is a fibration
\[ \xymatrix{
B\calT_X \ar[r] & B\map_W(C_X, T) \ar[d] \\
                & BF
}\]
Since we are using torsion-free coefficients, the Serre spectral sequence says that $H^*(B\calT_X) \to H^*(B\map_W(C_X, T))$ is the inclusion of the $F$-invariant subring.  But $\map_W(C_X, T)$ is abelian, so the action of $F$ is trivial.
\end{proof}

\begin{proposition}
$\calH_X \to (\calC_X \times BT)^{N_X}$ induces an isomorphism on $H^*(-)$.   
\label{propIsoFixedPoints}
\end{proposition}
\begin{proof}
In this proof we use the temporary notation $Y = (\calC_X \times BT)^{N_X}$.  Let $h_X: \calH_X \to \calC_X$ be the restriction of the Hitchin map and let $\tau: Y \to \calC_X$ be the projections.   Here is a picture:
\[\xymatrix{
\calH_X \ar[dr]_{h_X} \ar[rr]^{q_X} &         & Y \ar[dl]^\tau \\
                                    & \calC_X &     
}\]

It suffices to show that the map $ R\tau_* \ul{\Q}_{Y} \to Rh_{X*} \ul{\Q}_{\calH_{X}}$ is a quasi-isomorphism, and for this it suffices to show that the induced map $i\oii R\tau_* \ul{\Q}_{Y} \to i\oii Rh_{X*} \ul{\Q}_{\calH_{X}}$ is a quasi-isomorphism where $i: \spec \C \to \calC_X$ is the cover corresponding to $(C_X, \sigma)$. Further, because $\tau$ and $h$ are both locally trivial, 
\[i^*R\tau_* \ul{\Q}_{Y} \cong H^*( \spec \C \times_{\calC_X} Y) = H^*(BT^{W_X})\]
\[i^* Rh_{X*} \ul{\Q}_{\calH_{X}} \cong H^*(\spec \C \times_{\calC_X} \calH_X) = H^*(B\calT_X)\]
and the map between them is induced by the map $\spec \C \times_{\calC_X} Y \to \spec \C \times_{\calC_X} \calH_X$.  By part 2 of \cref{thmDGfacts} this is map $B\calT_X \to BT^{W_X}$ induced by evaluation at $\sigma$, and by \cref{isoFibers} the latter induces an isomorphism on $H^*(B-)$.
\end{proof}


\begin{corollary}
\hspace{1mm}
\label{corMoney}
Let $K_X$ be the kernel of the restriction $H^*(\calC \times BT) \to H^*((\calC_X \times BT)^{N_X})$.  If the pullback of $[\ol{\calC}_X]$ to $H^*( (\calC \times BT)^{N_X})$ is not a zero-divisor, then
\begin{itemize}
\item $q$ induces a surjection $q^*: H^*(\calC \times BT) \to H^*(\calH_\calC)$.
\item $\delta: \coprod_X \calH_X \to \calH_\calC$ induces an injection $H^*(\calH_\calC) \to H^*\left( \coprod_{X \in L(\calA)} \calH_X\right)$.
\item The kernel of $q^*$ is $\bigcap_X K_X$.
\end{itemize}
\end{corollary}
\begin{proof}
By \cref{lemInjection} and \cref{lemSurjection}, to prove (1) and (2) it suffices to prove the following 4 points: for every $X \in L(\calA)$,
\begin{enumerate}
\item $H^*(\calH_X)$ is concentrated in even degrees. 
\item $H^*(\calC \times BT) \to H^*(\calH_X)$ is surjective.
\item $[\ol{\calH}_X]$ is in the image of $q^*$.
\item $[\ol{\calH}_X]|_{\calH_X}$ is not a zero-divisor.
\end{enumerate}
$H^*(\calH_X) \cong H^*((\calC_X \times T)^{N_X})$, $(C_X \times T)^{N_X}$ is a $T^W$ torsor over $\calC_X$, and the cohomology $H^*(\calC_X)$ is concentrated in even degrees.  Thus if $H^*(BT^W)$ is concentrated in even degrees then so is $H^*(\calH_X)$.  But $T^W$ is a diagonalizable group, so $H^*(BT^W)$ is concentrated in even degrees.

We showed in section 5 that $H^*(\calC) \to H^*(\calC_X)$ is surjective.  Since $T^W$ is a diagonalizable subgroup of $T$, $H^*(BT) \to H^*(BT^W)$ is surjective as well.  Tensor products of surjective maps are surjective, so $H^*(\calC \times BT) \to H^*((\calC_X \times BT)^{N_X})$ is surjective.  By the previous proposition, the latter cohomology group surjects to $H^*(\calC \times BT)$.

Because $\calH_\calC \to \calC$ is a gerbe for $\calT$, a flat $\calM$-group, it is smooth.  Then $[\ol{\calH}_X]$ is the pullback of $[\ol{\calC}_X]$.  In particular it is in the image of $q^*$.  Since $q_X^*$ is an isomorphism, so $[\ol{\calC}_X]$ is a zero divisor in $H^*(\calH_X)$ \iffw it is a zero-divisor in $H^*((\calC_X \times BT)^{N_X})$.  

This completes the proof of (1) and (2).  (3) follows immediately from (2) and the previous proposition.
\end{proof}

To compute $H^*(\calH_\calC)$ it remains show that $[\ol{\calC}_X]$ is not a zero divisor in $(\calC_X \times BT)^{N_X}$ and to find the kernels $K_X$.  This will require a more explicit study of the spaces $(\calC_X \times BT)^{N_X}$.  


%
%

\begin{notation}
\label{crazyNotation}
If $G$ is a group, let $G^\vee$ be the character group of $G$ and $\cha G = \Q \otimes_\Z G^\vee$
We use additive notation for characters under tensor product.  If $E_T$ is a
$T$-torsor and $\chi \in G^\vee$, write $c_\chi E_T$ for the first chern class
of the line bundle obtained from $E_T$ by $\chi$. Extend $\Q$-linearly to
accommodate $\chi \in \cha G$. 

Let $x \in Q_X$ and $w \in W_X$.  For $\calB \in \irr \calA_X$, let $x_\calB \in \G_m$ be the image of $x$ under the composition $Q_X \to \aut_W T_\sigma C_X \to \G_m$, where the second map is projection to the $\G_m$ factor corresponding to $\calB$ (see \cref{pointStabalizers} for what this means).  For $\alpha$ a root, dual to a hyperplane $a_\alpha \in \calA_X$, let $\calB_\alpha$ be the irreducible component of $\calA_X$ including $a_\alpha$.  Let $\Phi^X_w$ be the set of positive roots $\alpha \in \Phi^+$ such that $w(\alpha) \in \Phi^-$ and such that $a_\alpha \in \calA_X$.  Define $\tilde{x}_w^X \in T$ by
 \[\tilde{x}_w^X =  \prod_{\alpha \in \Phi^X_w} \check{\alpha}(x_{\calB_\alpha})\]
Write $\Phi_w = \Phi^{\fk{t}}_w$ and $\tilde{x}_w = \tilde{x}^{\fk{t}}_w$.

There is a dual construction in cohomology.  Suppose given $\chi \in H^2(BT) \cong \cha G$.  We may identify $\chi$ with a $\Q$-linear sum of roots in $\fk{t}^\vee$.  Therefore there is a natural pairing  $\langle \check{\alpha}, \chi \rangle \in \Q$.  Define $\tilde{\chi}_w \in H^2(\calC, \Q)$ by  
\[ \tilde{\chi}_w = \sum_{\alpha \in \Phi_w} \langle \check{\alpha}, \chi \rangle a_\alpha \]
For $n \in N$, let $\tilde{x}_n^X = \tilde{x}_w^X$ and $\tilde{\chi}_n = \tilde{\chi}_w$ where $w$ is the image in $W$ of $n$.
\end{notation}

\begin{lemma}
Let $X \in L(\calA)$.  $\calR^n|_{\calC_X}$ is a $T$-torsor on $BQ_X$, so it corresponds to a homomorphism $Q_X \to T$.  This homomorphism sends $x \to \tilde{x}_n^X$.
\label{lemQxRaction}
\end{lemma}
\begin{proof}
$\calR^\alpha$ was obtained by inducing $I_{a_\alpha}^\vee$ from a $\G_m$-torsor to a $T$-torsor using the coroot $\check{\alpha}$.  Therefore to compute the image of $x \in Q_X$ in $T$ it is enough to understand how $x$ acts on $I_{a_\alpha}^\vee$, considered as a character of $Q_X$.  If $a_\alpha \not \in \calA_X$ then $x$ does nothing.  Otherwise, let $\calB_\alpha \in \irr \calA_X$ be the irreducible component including $a_\alpha$.
Then $x$ scales $I_{a_\alpha}$ by $x_\calB$.  Therefore $x$ acts on the $T$-torsor $R^{\alpha}_{\calC_X}$ by $\check{\alpha}(x_{\calB_\alpha})$ if $a \in \calA_X$, and by the identity otherwise.  This implies that $x$ acts on $\calR^n$ by right multiplication with $\tilde{x}^X_n$.  
\end{proof}

\begin{corollary}
$c_\chi \calR^w = \tilde{\chi}_w$
\label{corCharacteristicClassR}
\end{corollary}
\begin{proof}
By \cref{deltaInjectionC} it suffices to show that the two cohomology classes agree on strata.  \Cref{lemQxRaction} implies that $\left( \G_m \times_{T, \chi} \calR^w \right)|_{\calC_X}$ corresponds to the character $\psi: Q_X \to \G_m$ defined by $x \to \chi\left(\tilde{x}^X_w \right)$.  For $\alpha$ a root, let $a_\alpha$ be the dual hyperplane and $\calB_\alpha \in \calA_X$ be the irreducible component including $a_\alpha$.  For $\calB \in \irr \calA_X$, let $\calO(d \calB)$ be the corresponding degree-$d$ character of $Q_X$.  Then 
\[ \psi = \sum_{\alpha \in \Phi_w^X} \calO \left( \langle \check{\alpha}, \chi \rangle \calB_\alpha \right) \]
(Remember that we use additive notation for characters).  The first Chern class of this character is
\[ c_\chi \calR^w|_{\calC_X} = \sum_{\alpha \in \Phi^X_w} \langle \check{\alpha}, \chi \rangle a_\alpha|_{\calC_X} \]
Now note that $a_\alpha|_{\calC_X} = 0$ if $\alpha \in \Phi_w - \Phi^X_w$.  Therefore $c_\chi \calR^w|_{\calC_X} = \tilde{\chi}_w|_{\calC_X}$.
\end{proof}

\begin{proposition}
Let $V_X := \{ (x, y) \in Q_X \times T | (\forall w \in W_X) \,\,\, y^w = y \tilde{x}^X_w \}$.  Then $(\calC_X \times BT)^{N_X} = BV_X$, and $(\calC_X \times BT)^{N_X} \to \calC_X$ is induced by the projection $V_X \to Q_X$.  
\end{proposition}
\begin{proof}

$(\calC_X \times BT)^{N_X}$ is a $T^{W_X}$-gerbe over $\calC_X$.  As $\calC_X$ is a classifying space, this implies that $(\calC_X \times BT)^{N_X}$ is a classifying space itself.  It therefore suffices to show that $V_X$ is the automorphism group of a $\spec \C$ point of $(\calC_X \times BT)^{N_X}(\C)$.  Let's choose such a point.  As usual let $(C_X, \sigma)$ be a $\spec \C$ point of $\calC_X(\C)$.  To get a point of $(\calC_X \times BT)^{N_X}$ it remains to choose an $\calR|_{(C_X, \sigma)}$-twisted $N_X$-invariant $T$-torsor $\calJ$ on $\spec \C$.  Here $\calR^n|_{(C_X, \sigma)}$ denotes the pullback of $\calR^n$ along the map $\spec \C \to \calC_X$ defined by $(C_X, \sigma)$.  Without loss of generality, the underlying $T$-torsor of $\calJ$ is $T$ itself.  The equivariant structure is the data of isomorphisms $\phi_n: \calJ \to \calJ^n \otimes \calR^n|_{(C_X, \sigma)}$.  The pullback of $\calR$ to a point is trivializable, and after trivializing it we can take $\phi_n$ to be the the identity, so the choice of $\calJ$ is somewhat contentless. 

An automorphism $(C_x, \sigma, \calJ)$ is a pair $(x, y)$ where $x: C_X \to C_X$ is an automorphism and $y: x^*\calJ \to \calJ$ is an isomorphism of $\calR$-twisted $N$-equivariant $T$-torsors.  This last point is delicate: $x^*$ leaves $\calJ$ unchanged as a $T$-torsor, but $x$ does induce an automorphism $x^\#$ of $\calR^n|_{(C_X, \sigma)}$, and so alters the equivariant structure.  The requirement for $(x, y)$ to define an automorphism is that, for every $n \in N_X$, the diagram below commutes:
\[\xymatrix{
T \ar[rr]^{\phi_n} \ar[d]^y && T^n \otimes \calR^n_{(C_X, \sigma)} \ar[d]^{y \otimes x^\#} \\
T \ar[rr]^{\phi_n}         & & T^n \otimes \calR^n_{(C_X, \sigma)}
}\]

By \cref{lemQxRaction}, $x^\#$ is just multiplication by $\tilde{x}^X_n$.
Identify $y: T \to T$ with multiplication by some element of $T$, which we will
call $y$ by abuse of notation.  $\phi_n$ sends $1 \in T$ to some $(1, u) \in
T^w \otimes \calR^n$,
hence sends $a \in T$ to $(a^w, u)$.  Afterwards, $y \otimes x^\#$ sends this to $(a^wy, \tilde{x}_w^X u)$.  On the other hand, if we start with $a$ in the upper left diagram and go down then right, $a$ is sent first to $ay$ and then to $(a^wy^w, u)$.  Thus the commutative diagram reduces to the equation $y^w = y \tilde{x}^X_n$. 
\end{proof}

\begin{corollary}
\label{corStrataKernel}
\begin{enumerate}
\item $H^*(\calC_X \times BT) \to H^*((\calC_X \times BT)^{N_X})$ is a surjection with kernel generated by the degree-1 elements 
\[ \left(\chi - w \cdot \chi + \tilde{\chi}_w\right)|_{\calC_X} \]
as $\chi$ runs over $\cha T$ and $w$ runs over $W_X$. 
\item $[\calC_X]$ is not a zero-divisor in $H^*((\calC_X \times BT)^{N_X})$
\end{enumerate}
\end{corollary}
\begin{proof} 
$Q_X$ is an extension of $\aut_W T_\sigma C_X$ by a unipotent group, and
$\tilde{x}^X_n$ depends only on the image of $x$ in $\aut _W T_\sigma C_X$, so
it will suffice to study $\ol{V}_X \to \aut_W T_\sigma C_X \times T$ where
$\ol{V}_x$ is the subgroup of $\aut_W T_\sigma C_X \times T$ cut out by the
condition $y^w = y \tilde{x}^X_n$.  This is the kernel of the map $\aut_W
T_\sigma C_X \times T \to \prod_{w \in W} T$, defined by $(x, y) \to \prod_{w
\in W} (y^w)\oii y \tilde{x}^X_n$.  Therefore it is a diagonalizable group, and
its character group is the cokernel of the dual map $\psi: \bigoplus_{w \in W}
\cha T \to \cha (\aut_W T_\sigma C_X \times T)$.  $\psi$ decomposes as a sum of
maps $\psi_w: \cha T \to \cha(\aut_W T_\sigma C_X \times T)$.

In the course of proving \cref{corCharacteristicClassR} we showed that the character of $Q_X$ defined by $x \to \chi\left(\tilde{x}_n^X \right)$ is $\sum_{\alpha \in \Phi_n^X} \calO\left( \langle \check{\alpha}, \chi \rangle \calB_\alpha \right)$.  Therefore
\[ \psi_w(\chi) = \chi - \chi^w + \sum_{\alpha \in \Phi^n_X} \calO \left( \langle \check{\alpha}, \chi \rangle \calB_\alpha \right)\]
Statement 1 follows, and statement 2 is immediate from statement 1.

\end{proof}

We have now verified all the conditions of \cref{corMoney}.


\begin{proposition}
\label{propCohomologyHN}
Let $r(X, \chi, w) = \tilde{\chi}^X_w + \chi - \chi^{w}$, and let $I \subset H^*(\calC) \otimes_\Q H^*(BT)$ be the ideal generated by all expressions of the form $X \otimes r(X, \chi, w)$ where $X \in L(\calA)$,  $\chi \in \cha T = H^2(BT)$, and $w \in W_X$.
\begin{enumerate}
\item $I$ is the kernel of the surjection $q^*:  H^*(\calC) \otimes H^*(BT) \to H^*(\calH_\calC)$.  
\item The action of $W$ on $H^*(\calH_\calC)$ lifts to the action of $W$ on $H^*(\calC) \otimes H^*(BT)$ defined by 
\[ a \otimes 1 \cdot w = (a \cdot w) \otimes 1, \hspace{7mm} 1 \otimes \chi = \tilde{\chi}_w \otimes 1 + 1 \otimes \chi^w\]
and extending multiplicatively. 
\end{enumerate}
\end{proposition}
\begin{proof}

Write $K$ for the kernel of $q^*$.  
We will compute $K$ using the diagram below:
\[ \xymatrix{
H^*(\calC \times BT) \ar[r]^{\delta_1\,\,\,} \ar[d]^{q^*} & H^*( \coprod_{X \in L(\calA)} \calC_X \times BT) \ar[r]^{\gamma} & H^*(\coprod_{X \in L(\calA)} (\calC_X \times BT)^{N_X}) \ar[d]^\cong \\
H^*(\calH) \ar[rr]^{\delta_2}                 &                                              & H^*(\coprod_{X \in L(\calA)} \calH_X) 
}\]
As $\delta_1$ and $\delta_2$ are injections, $K = \delta_1\oii \ker \gamma$.  Let
$I_X \subset H^*(\calC, \Q) \otimes H^*(BT, \Q)$ be the ideal generated by the
elements $1 \otimes r(X, \chi, w)$ as $w$ runs over $W_X$ and $\chi$ runs over
$\cha T$, and write 
\[\gamma = \bigoplus_X \gamma_X: H^*(\calC_X \times BT) \to
H^*((\calC_X \times BT)^{N_X})\] 
\[\delta_1 = \bigoplus_X \delta_X: H^*(\calC \times BT) \to H^*(\calC_X \times BT) \]
We know from \cref{corStrataKernel}
that $\ker \gamma_X = \delta_X(I_X)$, and $\alpha \in K$ if and only if
$\delta_X(\alpha) \in \ker \gamma_X$ for all $X$, so $K = \bigcap_X \left( I_X
+ \ker \delta_X \right)$.

We first claim that for any $X \in L(\calA)$ and $w \in W_X$, $X \otimes r(X, \chi, w) \in K$.  For let $Y \in L(\calA)$.  If $X \not\leq Y$ then $X \in \ker \delta_Y$.  If $X \leq Y$ then $W_X \subseteq W_Y$, so $r(X, \chi, w) \in I_Y$.  Either way, $X r(X, \chi, w) \in \ker \delta_Y + I_Y$.  Thus $I \subset K$.

We will show the converse by induction on $L(\calA)$.  Let $P$ be an
upwards-closed sub-poset of $L(\calA)$, and suppose we already know that if $\alpha \in K$ with $\alpha$ supported on $\calC_P \times BT$, then $\alpha \in I$.  The base case (where $P = \{0\}$) follows from \cref{corStrataKernel}.  Now let $X \in L(\calA) \bs P$ such that $\{X\} \cup P$ remains a poset, and suppose that $\alpha$ is supported on $\left(\calC_X \cup \calC_P\right) \times BT$.  All elements of $H^*(\calC \times BT)$ with such a support are in the ideal $(X, X_1, \ldots, X_n)$ where $\{X_1, \ldots, X_n\}$ is a set of minimal elements of $P$.  Therefore we may write $\alpha = \alpha' + X \beta $, where $\alpha' \in (X_1, \ldots, X_n)$.  
Then $\delta_X(\alpha) = \delta_X(X\beta)$, so
\[\begin{array}{rcl}
0 &=& \gamma_X \delta_X(\alpha) \\
  &=& \gamma_X  \delta_X(X \beta) \\
  &=& \gamma_X\delta_X(X) * \gamma_X\delta_X(\beta) \\
  &=& \left( [\ol{\calH}_X]|_{\calH_X} \right) \gamma_X\delta_X(\beta)
\end{array}\]
In \cref{corMoney} we showed that $[\ol{\calH}_X]|_{\calH_X}$ is not a zero-divisor.  Therefore $\gamma_X\delta_X(\beta)= 0$, and $\beta \in \ker \delta_X + I_X$.  Write $\beta = \beta_0 + \beta_1$ where $\beta_0 \in \ker \delta_X$ and $\beta_1 \in I_X$.  We now have $\alpha = \alpha' + X\beta_0 +X\beta_1$.  $X \beta_1 \in I$, hence in $K$, so $\alpha' + X\beta_0 \in K$.  But $\alpha' + X\beta_0$ is supported on $P$, so by inductive hypothesis it is in $I$ as well.  This proves (1).

For (2), $\calH_\calC \to \calC \times BT$ is $N$-equivariant for the $\calR$-twisted action so it suffices to check that the twisted action on $H^*(\calC \times BT, \Q)$ is the one indicated in the statement.  The map $\calC \times BT \to \calC$ is $W$-equivariant, so $(a \otimes 1) \cdot w = (a \cdot w) \otimes 1$.  To determine $(1 \otimes \chi) \cdot w$ we must work a little harder.  Let $E_T$ be the pullback of $ET$ to $\calC \times BT$.  Then $1 \otimes \chi = c_\chi E_T$.  Therefore $(1 \otimes \chi) \cdot w = c_\chi \calR^w \otimes E_T^w$.  By \cref{corCharacteristicClassR}, $ c_\chi \calR^w \otimes E_T^w= \tilde{\chi}_w\otimes 1 + 1 \otimes \chi^w$.
\end{proof}

As $\calC \to \calM$ is the quotient map for the action of $W$ on $\calC$, its base change $\calH_\calC \to \calH$ is the quotient map for the action of $W$ on $\calH_\calC$.  We conclude

\begin{theorem}
\label{thmCohomologyH}
$H^*( \calH)$ is isomorphic to $\left( \Q[L^\mu(\calA)] \otimes H^*(BT) /I \right)^W$, where the ideal $I$ and the action of $W$ are as in the previous proposition.  The natural map $\Q[L^\mu(\calA)]^W \to H^*(\calH)$ is induced by the Hitchin map. 
\end{theorem}


\section{K-Theory}\label{chapK} 

In this section we compute $K_0\calC$ and identify $K_0\calM$ as a subring of the $W$-equivariant $K$-theory $K_W \calC$.  This is not very effective, since we do not compute $K_W \calC$.  Throughout this section, if $f: X \to Y$ is a map we write $f_\bul: K_0X \to K_0Y$ and $f^\bul: K_0Y \to K_0X$ for the associated maps on K-theory whenever they exist, and if $f$ is a closed immersion we reserve $\calF|_X$ for $f^\bul \calF$ (rather than for $f^*\calF$).  We begin with some lemmas about stratified stacks.

\begin{lemma}
\label{lemShortExact}
Let $i: \calY \into \calX$ be an lci closed immersion of Artin stacks with $K_0\calX = K^0\calX$, $K_0\calY = K^0\calY$.  Suppose $i^\bul i_\bul\calO_\calY$ is not a zero-divisor in $K_0 \calY$.  
\begin{enumerate}
\item The sequence
\[ 0 \to K_0 \calY \to K_0\calX \to K_0 \calX \bs \calY \to 0 \]
is exact.
\item The map $K^0 \calX \to K^0 \left( \calY \coprod \calX \bs \calY \right)$ is an injection.
\end{enumerate}
\end{lemma}
\begin{proof}
Let $N^*$ be the conormal bundle of $i$.  Define $\Lambda = \sum_i^n (-1)^i \wedge^i N^*$.  Since $i$ is lci, the Koszul resolution shows that $i^\bul i_\bul \alpha = \alpha \otimes \Lambda$.  In particular $i^\bul i_\bul = \Lambda$, so our assumption says that $\Lambda$ is not a zero divisor.  For (1) note that if $\beta \in K_0 \calY$ with $i_\bul\beta = 0$, then $0 = i^\bul i_\bul \beta = \beta \otimes \Lambda$ which implies $\beta = 0$.  For (2), suppose that $\alpha|_{\calY} =0$ and $\alpha_{\calX \bs \calY} = 0$.  By d\'{e}vissage, the latter equality implies $\alpha = i_\bul \beta$ for some $\beta \in K_0\calY$.  But then $i^\bul i_\bul\beta = 0$, which implies again that $\beta = 0$.
\end{proof}

\begin{definition}
Let $\{\calS_X\}_{X \in L}$ be a stratification of an Artin stack $\calS$ by a finite poset $L$.  The stratification is \emph{lci} if for every $X \in L$, the closed immersion 
\[ \calS_X \into \left( \calS - \ol{\calS}_X \right) \cup \calS_X\]
is a local complete intersection. 
\end{definition}

Let $L$ be a finite poset and $\calS$ be a stack with a chosen lci stratification by $L$.  Recall the abuse of notation $\calO_X = i_*\calO_{\ol{\calS}_X}$, where $i: \ol{\calS}_X \into \calS$ is the inclusion.  Assume further that $K_0\calS_X = K^0\calS_X$ for all $X$ and that $\calO_X|_{\calS_X}$ is not a zero-divisor in $K_0 \calS_X$.  These are essentially the same conditions that we put on cohomology in section 5.  

Essentially the same arguments as those for \cref{lemSurjection} give
\begin{lemma}
\label{lemKSurjection}
Let $\calS$ be a stack with lci stratification by a finite poset $L$, such that for all $X \in L$, $\calO_X|_{\calS_X}$ is not a zero-divisor.  Let $R \subset K_0 \calS$ be some subring, such that for all $X \in L$, the image of $R$ in $K_0 \calS_X$ is all of $K_0 \calS_X$.  Suppose further that $\calO_X \in R$ for all $X \in L$.  Then $R = K_0\calS$.
\end{lemma}

One can formulate a direct analogue of \cref{lemInjection} as well.  However, \cref{lemShortExact} (2) implies a somewhat stronger statement:
\begin{lemma}
\label{lemKInjection}
Let $\calS$ be a stack with lci stratification by a finite poset $L$, such that for all $X \in L$, $\calO_X|_{\calS_X}$ is not a zero-divisor.  Then the map $\delta: \coprod_{X \in L} \calS_X \to \calS$ induces an injection $\delta^\bul: K_0 \calS \into K_0\left( \coprod_{X \in L} \calS_X \right)$.
\end{lemma}

The assumption that $\calO_X|_{\calS_X}$ is a non-zero divisor in $K^0\calS_X$ is doing a lot of work in the previous statements, so we now look for a criterion that will guarantee this.  The main idea behind the next two facts, \cref{lemGroupProp} and \cref{corKstrata}, is due to Dima Arinkin.

\begin{lemma}
Let $G$ be an algebraic group, $S: \G_m \into G$ a one-parameter subgroup, $N^*$ a finite-dimensional representation of $G$, and $\Lambda = \sum_i (-1)^i \wedge^i N^*$ considered as an element in Grothendiek ring $KG$ of finite-dimensional representations of $G$.  Suppose that every eigenspace of $N^*|_{S}$ is of strictly positive degree.  Then $\Lambda$ is not a zero-divisor.
\label{lemGroupProp}
\end{lemma}
\begin{proof}
Let $n = \dim N^*$.  Decompose $N^*|_{\G_m}$ as a sum of characters: $N^*|_{\G_m} = \sum_i \calO(d_i)$, where $\calO(d)$ is the degree-$d$ character of $\G_m$.  If $V$ is a representation of $G$, let $\deg V$ be the highest degree of any eigenspace of $V|_{\G_m}$.  Then 
\[\deg \wedge^i N^* = \max \left\{ \sum_{j=1}^i d_{q_j} | 0 < q_1 < q_2 < \ldots < q_i \leq n \right\}\]
Because all the $d_q$ are strictly positive, $\deg \wedge^n N^* = d_1 + \ldots + d_n$ and $\deg \wedge^i N^* < d_1 + \ldots + d_n$ for all $i<n$.   In particular, $\Lambda|_{\G_m} \neq 0$. 

Let $x \otimes \Lambda = 0$.  Write $x = A-B$, where $A$ and $B$ are semisimple representations of $G$ with no common subquotient.  $K\G_m \cong \Z[u]$ has no zero divisors, and $\Lambda|_{\G_m} \neq 0$, so $A|_{\G_m} = B|_{\G_m}$.  In particular they have the same eigencharacters with the same  multiplicities.  Let $\Lambda^+ = \sum_i \wedge^{2i} N^*$ and $\Lambda^- = \sum_i \wedge^{2i+1} N^*$.  Then
\[ \left( A \otimes \Lambda^+ \right) \oplus \left( B \otimes \Lambda^- \right) =  \left( B \otimes \Lambda^+\right) \oplus \left( A \otimes \Lambda^-\right) \]
in $KG$.

Assume for contradiction that $A$ and $B$ are nonzero.  Let $s$ be the minimal degree of any eigenspace occurring in $\left( A \otimes \Lambda^+ \oplus B \otimes \Lambda^- \right)|_{\G_m}$.  Let $V$ be the smallest subrepresentation of $A \otimes \Lambda^+ \oplus B \otimes \Lambda^-$ that contains all eigenspaces for $\G_m$ of degree $s$.  $A|_{\G_m} = B|_{\G_m}$, and all degrees of eigenspaces in $\wedge^i N|_{\G_m}$ are strictly bigger than $0$ for $i>0$.  Therefore $V$ is contained in $A \otimes \wedge^0 N^* \cong A$.  Similarly let $W$ be the smallest subrepresentation of $B \otimes \Lambda^+ \oplus A \otimes \Lambda^-$ that contains all eigenspaces of degree $s$, then $W \subset B \otimes \wedge^0N^* \cong B$.  By the Jordan-Holder theorem, $V$ and $W$ share all the subquotients of $A \otimes \Lambda^+ \oplus B \otimes \Lambda^-$ that contain an eigenspace of degree $s$.  In particular they share at least one subquotient, so $A$ and $B$ share a subquotient, a contradiction.
\end{proof}

\begin{corollary}
\label{corKstrata}
Let $G$ be an affine algebraic group and $i: BG \into \calX$ be an l.c.i. closed immersion.  Let $N^*$ be the conormal bundle of $i$ and suppose there exists $S: \G_m \to G$ a one parameter subgroup such that every eigenspace of $N^*|_{\G_m}$ is of strictly positive degree.  Then $\calO_{BG}|_{BG}$ is not a zero-divisor in $K_0 BG$.  
\end{corollary}
\begin{proof}
The Koszul resolution implies $i^\bul i_\bul\calO_{\calY} = \sum_i (-1)^i \wedge^i N^*$.  The statement now follows from the previous lemma.  
\end{proof}

\begin{example}  Let $i': \spec \C \into \A^1$ be the inclusion of the origin.  Then $(i')^\bul i'_\bul \calO_{\spec \C}$ is $0$ in $K^0 pt$.  For $(i')^\bul i'_\bul\calO_{\spec \C}$ is computed by resolving $\C$ as a $\C[x]$ module and then pulling back and taking an alternating sum.  We resolve by $(x) \into \C[x]$, which pulls back to $\C \to C$.  The alternating sum is then $(i')^\bul i'_\bul\calO_{\spec \C} = \C - \C = 0 \in K^0 \spec \C$.  

On the other hand, let $i: B\G_m \into \left[ \A^1/\G_m \right]$ be the inclusion of the origin.  We identify quasicoherent sheaves on $B\G_m$ with $\G_m$-equivariant vector spaces and quasicoherent sheaves on $\left[ \A^1/\G_m \right]$ with $\G_m$-equivariant $\C[x]$-modules.  Thus we can calculate $i^\bul i_\bul \calO_{B\G_m}$ in essentially the same way as before, only this time we pay attention to degree, so $i^\bul i_\bul \calO_{B\G_m} = \C[0] - \C[1]$ where $\C[d]$ denotes the degree-$d$ character of $\G_m$.  In particular $i^\bul i_\bul \calO_{B\G_m}$ is nonzero, and the reason is that $\G_m$ acts on the conormal bundle $(x)/(x^2)$ with strictly positive degree.   
\end{example}

We are now ready to apply these results to $\calC$, with its stratification by $L(\calA)$.  The preimages of the strata of $\calC_X$ are linear spaces, and in particular are regular immersions.  Thus the stratification of $\calC$ is lci.  It remains only to check

\begin{lemma}
 $\calO_{\ol{\calC}_X}|_{\calC_X}$ is not a zero-divisor in $K_0\calC_X$.
\end{lemma}
\begin{proof}
Let $i: \calC_X \into \calC - \bigcup_{Y>X} \calC_Y$ be the inclusion, and
write $N^*$ for its conormal bundle.  Part 1 of \cref{corNormals} tells us that
$N^*$ is isomorphic to $T_\sigma C_X$ as a $Q_X$ representation.  Since $Q_X$
is a split extension of $\aut_W T_\sigma C_X$, it suffices (by
\cref{corKstrata}) to find a one-parameter subgroup $S:\G_m \to G$ so that
every eigenspace of $N^*|_{\G_m}$ is of strictly positive degree.  Under the
identification $\aut_W T_\sigma C_X \cong \G_m^{\# \irr \calA_X}$, the one
parameter subgroup $S*(\lambda) = (\lambda\oii, \lambda\oii, \ldots,
\lambda\oii)$ accomplishes this.  This shows (1).
\end{proof}

\begin{theorem}
The function $X \to \calO_{\ol{\calN}_X}$ extends to an isomorphism $\phi: \Z[L^{\mu}(\calA)] \cong H^*(\calC, \Z)$
\label{K0C}
\end{theorem}
\begin{proof}
The classes $\calO_a$ generate $K_0\calC_X$ for every $X$.  Therefore we can invoke \cref{lemKInjection} and \cref{lemKSurjection}.  Showing that the map is well-defined and injective follows an argument essentially identical to that of \cref{HN}. 
\end{proof}

Now we turn to $K_0\calM$.  We will first show that $K_0 \calM$ embeds in $K_0 [\calC/W]$.  In fact this is true in somewhat greater generality: all that matters is that $\calM$ is ``locally a finite GIT quotient of $\calC$''.

\begin{lemma}
\label{GITlemma}
Suppose given a Cartesian square of stacks
\[\xymatrix{ 
[Y/W] \ar[d]^b\ar[r]^T    & Y/W \ar[d]^a \\
\calY         \ar[r]^\tau & \calX
}\]
where $Y$ is a scheme, quasiprojective over $\C$, with action of a finite group $W$, and $a,b$ are both smooth and schematic.  Then 
\begin{enumerate}
\item $\tau^*$ and $\tau_*$ are exact.
\item $\tau_*\tau^* \cong Id$.
\item $\tau^\bul: K_0 \calX \to K_0 [\calY/W]$ is injective. 
\item $\alpha \in K_0 \calY$ is in the image of $\tau^\bul$ if any only if for every point $p: \spec \C \to \calY$ with image $r \in \calX(\spec \C)$, $\alpha|_p$ is in the image of $\rep G_r \to \rep G_p$ where $G_r$, $G_p$ are the residue gerbes of $r$ and $p$ respectively.
\end{enumerate}
\end{lemma}

\begin{proof}
The analogous fact for GIT quotient of schemes is known (and much more: see \cite{nevins} for example) so this will be just a matter of showing that they still hold equivariantly.  $Y \to X$ is the quotient map by a finite group, so $T^*$ is exact, $T_*$ is exact, and the natural map $Id \to T_*T^*$ is an isomorphism.  For concreteness, note that $T^*$ sends a sheaf on $X$ to a sheaf on $Y$ with $W$-equivariant structure, and $T_*$ sends a $W$-equivariant sheaf $\calF$ on $Y$ to $(f_*\calF)^W$.  These are exact because $f$ is flat and affine and $W$ is finite.  

The vertical maps are flat, so exactness of $T^*$ and $T_*$ implies (1).  We can check that the natural map $\calF \to \tau_*\tau^*\calF$ is an isomorphism after pulling back along $a$.  As $a$ is flat and $T_*, T^*, \tau_*, \tau^*$ are all exact, so there are natural isomorphisms
\[ \begin{array}{rcl}
a^* \tau_* \tau^* \calF &=&  T_* b^* \tau^* \calF \\
                        &=&  T_*T^* a^*
\end{array}\]
and this identifies the isomorphism $a^* \calF \to T_*T^*a^*\calF$ with the pullback along $a$ of $\calF \to \tau_*\tau^*\calF$.  
This shows (2), and (3) follows easily from (1) and (2).  

The ``only if'' part of (4) is clear.  To prove the ``if'' direction, it would suffice to show that if $\calF$ is an honest sheaf on $\calY$ such that $p^*\calF$ is in the image of $\rep G_r \to \rep G_p$ for all $p \in \calY(\spec \C)$, then $\calF$ descends to $\calX$.  If $\calF$ is such a sheaf we may consider $b^*\calF$.  Let $y \in Y(\spec \C)$ with stabilizer $H \subset W$.  The image $\bar{y}$ of $y$ in $[Y/W]$ is $BH$ and the diagram above restricts to a Cartesian square 
\[\xymatrix{
BH \ar[r] \ar[d] & \spec \C \ar[d] \\
BG_p \ar[r]      & BG_r 
}\]
In particular $\bar{y}^*b^*\calF$ is a trivial representation of $H$.  Corollary 2.6 of \cite{nevins}  then implies that $b^*T^*T_*\calF \to b^* \calF$ is an isomorphism.  This is the pullback under $b^*$ of the natural map $T^*T_*\calF \to \calF$, so the latter map is an isomorphism.  In particular $\calF$ descends to $\calX$.
\end{proof}

For $[X] \in L(\calA)/W$, let $\calC_{[X]} = \coprod_{Y \in [X]} \calC_Y$.  Then $\calC_{[X]}$ is $W$-stable and $\calC/W$ is stratified by the spaces $[\calC_{[X]}/W]$. Write $\calM_{[X]}$ for the image of $\calC_{[X]}$ in $\calM$. 

\begin{lemma}
\label{strataMaps}
$[\calC_{[X]}/W] \cong B\left(Q_X \rtimes \stab(X) \right)$, and the diagram below commutes:
\[\xymatrix{
\calC_{[X]} \ar[r] \ar[d]^\cong & [\calC_{[X]}/W] \ar[d]^\cong \ar[r] & \calM_{[X]} \ar[d]^\cong \\
\coprod_{Y \in [X]} BQ_X \ar[r] & B\left( Q_X \rtimes \stab(X)\right) \ar[r] & B \left(  Q_X \rtimes \stab(X)/\fix(X)\right)
}\]
\end{lemma}

\begin{proof}
$\calC_{[X]}$ is the disjoint union $\coprod_{Y \in [X]} \calC_Y$, and $W$ acts transitively on the summands.  The stabilizer of the summand $\calC_X$ is $\stab(X)$.  Therefore $[\calC_{[X]}/W] \cong [\calC_X/\stab(X)]$.  An $S$-point of $[\calC_{[X]}/W]$ is a triple $(C, E, \bar{\sigma})$ where $C \to S$ is a cameral cover of $S$, $E \to S$ is a $\stab(X)$-torsor, and $\bar{\sigma}: E \to C$ is a $W_X$-fixed, $\stab(X)$-equivariant map over $S$.  The map from $\calC_X$ sends a pointed cover $(C, \sigma) \to S$ to the triple $(C, \stab(X) \times S, \sigma')$ where $\sigma': \stab(X) \times S \to C$ is defined by $\sigma'(n, s) = n \cdot \sigma(s)$, and the map to $\calM_{[X]}$ sends $(C, E, \bar{\sigma})$ to $C$.

$[\calC_{[X]}/\stab(X)]$ is covered by $\calC_X$, and is therefore a classifying space, so to identify it we need only compute the automorphisms of a $\C$-point.  Let $\sigma_0$ be the $W_X$-fixed point of $C_X$.  Then $(C_X, \stab(X), \sigma_0')$ is a $\C$-point of $[\calC_{[X]}/W]$.  An automorphism of this triple is a pair $(\alpha, n)$ where $n \in \stab(X)$, $\alpha \in \bar{Q}_X$, and $n \cdot \sigma_0 = \alpha(y)$.  The group of such automorphisms is an extension of $\stab(X)$ by $Q_X$, and we can define a splitting as follows:  Recall the splitting $\rho: \stab(X)/W_X \to \bar{Q}_X$ from \cref{pointStabalizers}.  Precomposing with the quotient map $\stab(X) \to \stab(X)/W_X$ gives a map $\psi: \stab(X) \to \bar{Q}_X$, and $y \to (\psi(y), y) \in \bar{Q}_X \times \stab(X)$ splits the extension $1 \to Q_X \to \aut( C_X, \stab(X), \sigma_0') \to \stab(X) \to 1$.  
\end{proof}

\Cref{GITlemma} and \cref{strataMaps} together imply

\begin{theorem}
\label{KM}
$K_0\calM$ is the subring of $K_0[\calC/W]$ consisting of all classes $\alpha$ such that for all $[X] \in L(\calA)/W$, $\alpha|_{[\calC_{[X]}/W]}$ is in the image of $Rep(Q_X \rtimes \stab(X)/\stab(W_X)) \to Rep(Q_X \rtimes \stab(X))$
\end{theorem}

\section{Integral cohomology}\label{chapIntegral} 

Cameral covers have interesting finite-group monodromy (coming from the $\stab(X)/W_X$ factor of $\bar{Q}_X$).  Cohomology using coefficients where $|W|$ is invertible cannot detect this, so our computations are incomplete.  In this section we give a partial description of $H^*(\calM, \Z)$ which is enough to address to address at least one natural question: whether a given cameral cover is a degeneration of a less-ramified one.    

For most of this section we work with the following setup.  Let $X \in L(\calA)$ and  $\calX = \left( \calM - \ol{\calM}_X \right) \cup \calM_X \subset \calM$.  Similarly let $\calY = \left( \calC - \ol{\calC}_X \right) \cup \calC_X \subset \calC$, let $\calU = \calX - \calM_X$ and $\calV = \calY - \calC_X$.  Thus there is a commutative diagram
\[\xymatrix{
\calC_X \ar[r]\ar[d]^{\pi_c} & \calY \ar[d]^\pi & \calV \ar[l] \ar[d]^{\pi_o} \\
\calM_X \ar[r]               & \calX            & \calU \ar[l]
}\]
where the righthand square is Cartesian, the horizontal maps on the left are closed immersions, the horizontal maps on the right are open immersions, and the rows are decompositions of $\calY$ and $\calX$ respectively into a closed piece and an open piece.  

Our goal is to study the cohomology of $\calX$ using the gysin sequence.  The problem is that the gysin sequence is not natural for non-smooth maps such as $\pi$.  More explicitly, the map $\phi$ in the diagram
\begin{equation}
\label{gysinDiagram}
\xymatrix{
H^{*-r}(\calC_X, \Z) \ar[r]             & H^*(\calY, \Z) \ar[r]                      & H^*( \calV, \Z) \ar[r]                   & H^{*-r+1}(\calC_X, \Z) \\
H^{*-r}(\calM_X, \Z) \ar[r]\ar[u]^\phi  & H^*(\calX, \Z) \ar[u]^{\pi^*} \ar[r] & H^*( \calU, \Z) \ar[r] \ar[u]^{\pi_o^*} & H^{*-r+1}(\calM_X, \Z) \ar[u]^\phi
}
\end{equation}
cannot be determined solely from $\pi_c^*$. 

\begin{proposition}
\label{ramificationNumber}
$\phi(\alpha) = \#(\stab(X)/W_X) \cdot \pi_c^*\alpha$
\end{proposition}

We have been told that there is an easy proof of \cref{ramificationNumber} using homology.  We have decided to include the following long argument anyway, for the benefit of any readers who (like the author) are not familiar with homology of stacks or simplicial spaces.  This argument requires refinements of \cref{smoothCover} and \cref{kPoints}.

\begin{lemma}
\label{eqCover}
Let $T$ be the torus of $W$-equivariant automorphisms of $\fk{t}$ and have it act on $\fk{q} = \fk{t}/W$ in such a ways as to make $\fk{t} \to \fk{q}$ be $T$-equivariant.  Then the cover $\rho: \fk{q} \to \calM$ factors through the quotient stack $[\fk{q}/T]$.  Further, it induces an isomorphism from $T$ to the maximal torus of $\bar{Q}_0$.
\end{lemma}
\begin{proof}
Given a $T$-torsor $E_T$ together with a $T$-equivariant map $E_T \to \fk{q}$, pull back $\fk{t} \to \fk{q}$ to get a $T$-equivariant cameral cover $\bar{C} \to E_T$.  Since $\fk{t} \to \fk{q}$ is $T$-invariant, $\bar{C} \to E_T$ gets a $T$-equivariant structure, letting us descend to $S$.  This defines a map $[\fk{q}/T] \to \calM$ factoring $\rho$.  Suppose that $S = \spec \C$, $E_T = T$, and $T \to \fk{q}$ sends every point of $T$ to the origin.  The fiber of $\fk{s} \to \fk{q}$ over the origin $0$ is exactly the cameral cover $C_0 \to \spec \C$, and the $T$-equivariant structure on $T \times C_0$ obtained by pulling back $\fk{t} \to \fk{q}$ is exactly the action of $T$ on $C_0$ considered as a subscheme of $\fk{t}$.  We already identified this with the action of the maximal torus of $\bar{Q}_0$ in \cref{pointStabalizers}.
\end{proof}  

\begin{lemma}
The induction map $\calM^X \to \calM$ is \'{e}tale.
\end{lemma}
\begin{proof}
First we show that the map is schematic and finite type.  $\calM^X \times_\calM \calC = W/W_X \times \calC^X$ with the natural $W$-action.  As shown in \cref{refCohomology} the restriction of $W/W_X \times \calC^X \to \calC$ to the component corresponding to each coset in $W/W_X$ is an isomorphism onto its image, so $W/W_X \times \calC^X \to \calC$ is schematic and finite type.  Now let $f: S \to \calM$ be a map from a scheme corresponding to a spectral cover $C \to S$.  Then the pullback $C'$ of $W/W_X \times \calC^X \to \calM$ along $f$ is a schematic and finite type over $S$, since $\calC \to \calM$ is schematic and finite-type.  $S \times_\calM \calM_X  \cong C'/W$ so it is a finite-type scheme over $S$ as well.  

Now we can show that $\calM^X \to \calM$ is \'{e}tale using the formal criterion.  Let $C \to \spec R$ be a cameral cover and suppose $C \times_{\spec R} \spec R/I = \ind_{W_X}^W C'$ where $I$ is a square-zero ideal and $C'$ is a $(W_X, \fk{t})$-cameral cover of $\spec R/I$.  $C$ and $\ind_{W_X}^W C'$ are topologically identical.  In particular $C$ is a disjoint union of $R$-schemes indexed by $W/W_X$.  The $W$-action on $C$ can now be used to show that all the summands of $C$ are isomorphic and $C = \ind_{W_X}^W C''$ where $C''$ is the summand containing $C'$.  This shows formal smoothness.  $C''$ is uniquely determined (being the only union of components of $C$ that contains $C'$) so it shows formal \'{e}taleness as well.
\end{proof} 

The next few arguments involve some notions from homotopy theory.  To keep the notation readable, in this section we will write $\A \calF$ in place of $|\calF|$ for $\spec \sym^\bul \calF^\vee$.  If $A \to B$ is an inclusion of Artin stacks, write $\frac{B}{A}$ for the cofiber.  Let $Th(\calF)$ denote the Thom space $\A\calF/(\A\calF - S)$.  We interpret the cofiber of a map of stacks, and the Thom space of a sheaf on a stack, via complex realizations of simplicial varieties.  This is a bit unsatisfactory since it depends on the choice of cover, but as we are only interested in computing cohomology this will not matter.

If $i: \calZ \into \calS \gets \calU$ is a stratification of a smooth stack $\calS$ by a smooth closed substack $\calZ$ of pure codimension $r$ and its open complement $\calU$ then the gysin sequence $H^{*-r}(\calZ, \Z) \to H^*(\calS, \Z) \to H^*(\calU, \Z) \to H^{*-r+1}(\calZ, \Z)$ has a geometric interpretation involving the Thom space $Th(N_i)$.  Namely the purity theorem implies that $\frac{\calS}{\calU}$ is homotopic to $Th(N_i)$, while the Thom isomorphism identifies $\tilde{H}^*(Th(N_b), \Z)$ with $H^{*+r}( BT, \Z)$.  Under these identifications, the gysin sequence is merely the cohomology of the cofiber sequence 
\[ \calX \to \calU \to Th(N_i) \to \Sigma X \to \ldots \]

Just like the gysin sequence, this cofiber sequence is natural in smooth maps.  Further, if $\calS$ is a vector bundle and $\calZ$ is its zero section then the purity theorem becomes tautological, while if $\calS$ is an affine space and $\calZ$ is a linear subspace then the Thom isomorphism become tautological as well.  Our strategy is to use smooth covers of $\calM$ and $\calC$ by affine spaces to reduce to this tautological case.

\begin{lemma}
Let $(V, \fk{s})$ be an essential reflection arrangement and let $\fk{q} = \fk{s}/V$.   Let $T$ be the torus of $V$-equivariant linear automorphisms of $\fk{s}$, and let $T$ act on $\fk{q}$ so that $\fk{s} \to \fk{q}$ is $T$-equivariant.  Then the diagram below commutes, where $\psi$ is multiplication by $\#V$
\[\xymatrix{
H^*([\fk{s}/T], \Z)        \ar[r] & H^{*-r+1}(BT, \Z)  \ar[r]             & H^{*+1}([\fk{s}/T] - BT, \Z) \\
H^*([\fk{q}/T], \Z) \ar[u] \ar[r] & H^{*-r+1}(BT, \Z)  \ar[r] \ar[u]^\psi  & H^{*+1}([\fk{q}/T] - BT, \Z) \ar[u]
}\]
\end{lemma}

\begin{proof}
Denote by $a: BT \to \fk{s}/T$ and $b: BT \to \fk{q}/T$ the inclusion of the origin.  Consider the diagram of cofiber sequences
\[\xymatrix{
[\fk{s}/T] - BT \ar[r] \ar[d]^{\Phi} & [\fk{s}/T] \ar[r]\ar[d] & Th(N_a) \ar[r]\ar[d]^\Phi & \Sigma \left( [\fk{s}/T] - BT\right) \ar[d] \\
[\fk{q}/T] - BT \ar[r]               & [\fk{q}/T] \ar[r]       & Th(N_b) \ar[r]          & \Sigma ([\fk{q}/T] - BT)        \\
}\]
This is a diagram of spaces over $BT$ and the Thom isomorphism is a map of $H^*(BT, \Z)$ modules, so $\psi$ is $H^*(BT, \Z)$ linear.  Thus to specify $\bar{\phi}$ it suffices to compute $\psi(1)$.  In terms of Thom spaces, we are trying to compute the image in $H^*(Th(N_a), \Z)$ of the Thom class $u_b$ of $Th(N_b)$.  

If we identify $\fk{s}$ with the normal bundle of the origin $N_{0_{\fk{s}}}$ then the purity isomorphism $\frac{\fk{s}}{(\fk{s}-0)} \cong Th(N_{0_{\fk{s}}})$ is tautological.  Now the key point is that this still works once we take the stack quotient by $T$: $[\fk{s}/T]$ is naturally isomorphic, as a vector bundle on $BT$, to the normal bundle of $BT \into [\fk{s}/T]$.  This makes the purity isomorphism $[\fk{s}/T]/([\fk{s}/T] - BT) \cong T(N_a)$ tautological as well, so $\Phi$ is the map induced by the nonlinear map of vector bundles $[\fk{s}/T] \to [\fk{q}/T]$ over $BT$.  

Let $p: \spec \C \to BT$ be a point.  The fiber of $N_a$ over $p$ identifies with $\fk{s}$ and the fiber of $N_b$ over $p$ identifies with $\fk{q}$.  Then we get a commutative diagram
\[\xymatrix{
\fk{s}/(\fk{s}-0) \ar[r]\ar[d] & \fk{q}/(\fk{q}-0) \ar[d] \ar[r] & \spec \C \ar[d]^p \\
Th(N_a) \ar[r]                 & Th(N_b)                  \ar[r] & BT
}\]
where $\fk{s}/(\fk{s}-0)  \to \fk{q}/(\fk{q}-0) $ is induced by $\fk{s} \to \fk{q}$.  In particular it is an orientation-preserving map of spheres of degree $\#V$, and sends the orientation class of $\fk{q}/(\fk{q}-0)$ to $\#V$ times the orientation class of $\fk{s}/(\fk{s}-0)$.  

The Thom class $u_\calF$ of a locally free sheaf $\calF$ is uniquely characterized by a fiberwise condition: if $F$ is a fiber of $\A \calF$, then there is a map $F/(F-0) \to Th(\calF)$, and $u_\calF$ is the unique element restricting to the orientation class in $H^*(F/(F-0), \Z)$ for all fibers $F$.  Thus we must have $\psi(1) = \#V$.
\end{proof}

\begin{proof}[Proof of \cref{ramificationNumber}.]

Let $\fk{s} = \fk{t}/X$ and let $T$ be the torus of $W_X$-equivariant linear automorphisms of $\fk{s}$.  Then $\calM^X$ is smoothly covered by $\fk{q} := \fk{s}/W_X$ and this map factors through the stacky quotient $[\fk{q}/T]$, where the $T$ action on $\fk{q}$ is such as to make $\fk{s} \to \fk{q}$ $T$-equivariant.  Note that $\fk{q} \times_\calM \calC = W/W_X \times \fk{s}$.  By \cref{eqCover} we obtain a Cartesian square as below, where the horizontal maps are smooth:
\[\xymatrix{
W/W_X \times [ \fk{s}/T] \ar[r] \ar[d] & \calY \ar[d] \\
             [ \fk{q}/T] \ar[r]        & \calX
}\]
Restricting to the origin of $\fk{s}$ gives a map $BT \to \calC_X$ which is an isomorphism onto the $T$ factor of $Q_X$.  Similarly, restricting to the origin of $\fk{q}$ gives a map $BT \to \calM_X$ which is an isomorphism onto the $T$ factor of $\bar{Q}_X$.  The gysin sequence is natural for smooth maps so we obtain a commutative diagram as below, where all the maps are the obvious ones except for $\phi$ and $\bar{\phi}$:
\[\xymatrix{
H^{*+r}(W/W_X \times BT, \Z) \ar[rr]                    &                                         & H^*(W/W_X \times [\fk{s}/T], \Z)   &                        \\
                                           &H^{*+r}([W/W_X] \times \calC_X, \Z) \ar[ul]^\cong\ar[rr] &                       & H^*(\calY, \Z) \ar[ul] \\
H^{*+r}(BT, \Z) \ar[uu]^{\bar{\phi}}\ar[rr]&                                             & H^*([\fk{q}/T], \Z)\ar[uu] &     \\
                                           &H^{*+r}(\calM_X, \Z)\ar[ul]^\cong\ar[uu]^\phi\ar[rr]&              & H^*(\calX, \Z) \ar[ul]\ar[uu] }\]
Thinking of $H^{*+r}(W/W_X \times BT, \Z)$ as $\sum_{wW_x \in W/W_x} H^*(BT, \Z)$ it suffices to show that the vertical map in the back left of this diagram is $\#\stab(X)/W_X$ in each summand.  Since the map is $W$-equivariant we need only consider the map to a single summand $H^*{*+r}(BT, \Z) \to H^{*+r}(BT, \Z)$.  The claim now follows from the previous lemma.
\end{proof}

%

\begin{definition}
If $C \to S$ is a cameral cover in $\calM_X(S)$, call $C$ \emph{smoothable} if there exists a connected variety $B$ with $\C$-points $b_0$ and $b_1$ together with a cameral cover $\bar{C} \to B \times S$ such that the fiber $\bar{C} \times_B b_1 \to S$ is unramified and $\bar{C} \times_B b_0 \to S$ is isomorphic to $C \to S$.  Call $C$ \emph{somewhat smoothable} if there exists such a family with $\bar{C} \times_B b_1 \to S$ in $\calM_Y$ for $Y<X$.
\end{definition}

If $C \to S$ is smoothable then the corresponding map $S \to \calM$ is homotopic to one factoring through $BW$, so all characteristic classes of $C \to S$ must be $\#W$-torsion.  Using \cref{ramificationNumber} we get a subtler statement.

\begin{corollary}
\label{smoothingObstruction}
Let $X \in L(\calA)$ be codimension $r$ and let $C \to S$ be a cameral cover in $\calM_X(S)$ and let $p: C^{red} \to C$ be the reduced subscheme of $C$ with inclusion $i^{red}: C^{red} \into C$.  If $C$ is somewhat smoothable then the $r^{th}$ chern class of the conormal bundle $N^*_{i^{red}}$ is $\#(\stab(X)/W_X)$-torsion.  
\end{corollary}
\begin{proof}
Let $f: S \to \calM$ be the map corresponding to $C$.  $C^{red} \times_S C \to C^{red}$ has a section $\sigma = Id \times_S i^{red}$.  This gives a map $g: C^{red} \to \calC$.  
\[\xymatrix{
C^{red} \ar[d]^p \ar[r]^g & \calC \ar[d]^\pi \\
S                \ar[r]^f & \calM
}\]  
Since $C$ is somewhat smoothable, $f$ is homotopic to a map factoring through $\calU$.  Thus $f^*[\calM_X] = 0$.  The previous lemma shows that $\pi^*[\calM_X] = m [\calC_X]$, where $m = \#(\stab(X)/W_X)$, so $g^* [\calC_X]$ is $m$-torsion.  Let $i: \calC_X \into \calY$ be the inclusion.  Let $X = a_1\cap a_2 \cap \ldots \cap a_r$, then \cref{funClass} implies that $[\calC_X] = \bigotimes_{i=1}^r c_1 I_{a_i} = c_r \bigoplus_{i=1}^r I_{a_i}$.  Since $\calC_X$ is the transverse intersection of the $\ol{\calC}_{a_i}$, $i^*\bigoplus_{i=1}^r I_{a_i} = i^*\bigoplus_{i=1}^r N^*_{\ol{\calC}_{a_i}} = i^*\bigoplus_{i=1}^r N_i^*$.  We know (from \cref{corNormals}) that $i^*N_i^*$ is the conormal bundle to the universal section of $\calC_X$.  The pullback under $p$ of the conormal bundle of the universal section is the conormal bundle of the section $Id \times_S i^{red}: C^{red} \to C^{red} \times_S C$.  But this is the normal bundle of $i^{red}$. 
\end{proof}

\begin{example}
Let $\calL$ be a line bundle on $S$.  Consider $\bar{S} := \spec_S \sym^\bul \calL/\sym^n \calL$.  This is a spectral cover of $S$, which is totally ramified everywhere. Let $C \to S$ be the $(\Sigma_n, \C^{\oplus n})$ cameral cover corresponding to $\bar{S}$.  $C$ is somewhat smoothable \iffw $\bar{S}$ is a degeneration of a spectral cover $\bar{S}' \to S$ with no points of total ramification.

We have $C = Co(\Sigma_n, \C^{\oplus n}) \times_{\G_m} \calL$. This admits a section $\sigma: S \to C$ identifying $S$ with $C^{red}$.  Under this identification, the conormal bundle $N^*_{i^{red}}$ identifies with $\calL^{\bigoplus n-1}$.  Applying \cref{smoothingObstruction}, we see that $\bar{S}$ is a degeneration of a spectral cover $\bar{S}' \to S$ with no points of total ramification only if $c_1 \calL$ is $n!$-torsion.
\end{example}

\end{document}